\documentclass[a4paper,leqno,12pt]{amsart}
\usepackage{amsfonts,amssymb,verbatim,amsmath,amsthm,latexsym,textcomp,amscd}
\usepackage{latexsym,amsfonts,amssymb,epsfig,verbatim}
\usepackage{amsmath,amsthm,amssymb,latexsym,graphics,textcomp}
\usepackage{graphicx}
\usepackage{color}
\usepackage{url}
\usepackage{enumerate}
\usepackage{centernot}
\usepackage[mathscr]{euscript}
\input xy
\xyoption{all}
\begin{document}
\theoremstyle{plain}

%
\theoremstyle{plain}
\swapnumbers
    \newtheorem{thm}[figure]{Theorem}
    \newtheorem{prop}[figure]{Proposition}
    \newtheorem{lemma}[figure]{Lemma}
    \newtheorem{keylemma}[figure]{Key Lemma}
    \newtheorem{cor}[figure]{Corollary}
    \newtheorem{fact}[figure]{Fact}
    \newtheorem{subsec}[figure]{}
    \newtheorem*{propa}{Proposition A}
    \newtheorem*{thma}{Theorem A}
    \newtheorem*{thmb}{Theorem B}
    \newtheorem*{thmc}{Theorem C}
    \newtheorem*{thmd}{Theorem D}
\theoremstyle{definition}
    \newtheorem{defn}[figure]{Definition}
    \newtheorem{example}[figure]{Example}
    \newtheorem{examples}[figure]{Examples}
    \newtheorem{notn}[figure]{Notation}
    \newtheorem{summary}[figure]{Summary}
\theoremstyle{remark}
        \newtheorem{remark}[figure]{Remark}
        \newtheorem{remarks}[figure]{Remarks}
        \newtheorem{warning}[figure]{Warning}
    \newtheorem{assume}[figure]{Assumption}
    \newtheorem{ack}[figure]{Acknowledgements}
\renewcommand{\thefigure}{\arabic{section}.\arabic{figure}}
%
%
\newenvironment{myeq}[1][]
{\stepcounter{figure}\begin{equation}\tag{\thethm}{#1}}
{\end{equation}}
\newcommand{\myeqn}[2][]
{\stepcounter{figure}\begin{equation}
     \tag{\thethm}{#1}\vcenter{#2}\end{equation}}
\newcommand{\mydiag}[2][]{\myeq[#1]\xymatrix{#2}}
\newcommand{\mydiagram}[2][]
{\stepcounter{figure}\begin{equation}
     \tag{\thethm}{#1}\vcenter{\xymatrix{#2}}\end{equation}}
\newcommand{\mysdiag}[2][]
{\stepcounter{figure}\begin{equation}
     \tag{\thefigure}{#1}\vcenter{\xymatrix@R=10pt@C=20pt{#2}}\end{equation}}
\newcommand{\mytdiag}[2][]
{\stepcounter{figure}\begin{equation}
     \tag{\thethm}{#1}\vcenter{\xymatrix@R=15pt@C=25pt{#2}}\end{equation}}
\newcommand{\myrdiag}[2][]
{\stepcounter{figure}\begin{equation}
     \tag{\thethm}{#1}\vcenter{\xymatrix@R=20pt@C=10pt{#2}}\end{equation}}
\newcommand{\myqdiag}[2][]
{\stepcounter{figure}\begin{equation}
     \tag{\thefigure}{#1}\vcenter{\xymatrix@R=20pt@C=22pt{#2}}\end{equation}}
\newcommand{\myfigure}[2][]
{\stepcounter{figure}\begin{equation}
     \tag{\thefigure}{#1}\vcenter{#2}\end{equation}}
\newcommand{\mywdiag}[2][]
{\stepcounter{figure}\begin{equation}
     \tag{\thefigure}{#1}\vcenter{\xymatrix@R=20pt@C=12pt{#2}}\end{equation}}
\newcommand{\myzdiag}[2][]
{\stepcounter{figure}\begin{equation}
     \tag{\thethm}{#1}\vcenter{\xymatrix@R=5pt@C=20pt{#2}}\end{equation}}
%
\newenvironment{mysubsection}[2][]
{\begin{subsec}\begin{upshape}\begin{bfseries}{#2.}
\end{bfseries}{#1}}
{\end{upshape}\end{subsec}}
\newenvironment{mysubsect}[2][]
{\begin{subsec}\begin{upshape}\begin{bfseries}{#2\vsn.}
\end{bfseries}{#1}}
{\end{upshape}\end{subsec}}
\newcommand{\supsect}[2]
{\vspace*{-5mm}\quad\\\begin{center}\textbf{{#1}}\vsm.~~~~\textbf{{#2}}\end{center}}
\newcommand{\sect}{\setcounter{figure}{0}\section}
%
%
\newcommand{\wh}{\ -- \ }
\newcommand{\wwh}{-- \ }
\newcommand{\w}[2][ ]{\ \ensuremath{#2}{#1}\ }
\newcommand{\ww}[1]{\ \ensuremath{#1}}
\newcommand{\www}[2][ ]{\ensuremath{#2}{#1}\ }
\newcommand{\wwb}[1]{\ \ensuremath{(#1)}-}
\newcommand{\wb}[2][ ]{\ (\ensuremath{#2}){#1}\ }
\newcommand{\wref}[2][ ]{\ (\ref{#2}){#1}\ }
\newcommand{\wwref}[2]{\ (\ref{#1})-(\ref{#2})\ }
%
%
\newcommand{\hsp}{\hspace*{9 mm}}
\newcommand{\hs}{\hspace*{5 mm}}
\newcommand{\hsn}{\hspace*{1 mm}}
\newcommand{\hsm}{\hspace*{2 mm}}
\newcommand{\vsn}{\vspace{1 mm}}
\newcommand{\vs}{\vspace{5 mm}}
\newcommand{\vsm}{\vspace{3 mm}}
\newcommand{\vsp}{\vspace{9 mm}}
\newcommand{\lra}[1]{\langle{#1}\rangle}
\newcommand{\xra}[1]{\xrightarrow{#1}}
\newcommand{\lin}[1]{\{{#1}\}}

\newcommand{\blue}[1]{\textcolor{blue}{{#1}}}
\newcommand{\vare}{\varepsilon}
\newcommand{\hy}[2]{{#1}\text{-}{#2}}
\newcommand{\Al}{\mbox{\sf Alg}}
\newcommand{\Md}{\mbox{\sf Mod}}
\newcommand{\Alg}[1]{\hy{#1}{\Al}}
\newcommand{\Mod}[1]{\hy{#1}{\Md}}
\newcommand{\ModE}{\Mod{\bE}}
%
%
\newcommand{\Arr}{\operatorname{Arr}}
\newcommand{\cof}{\operatorname{cof}}
\newcommand{\ev}{\operatorname{ev}}
\newcommand{\F}{\operatorname{Fun}}
\newcommand{\holim}{\operatorname{holim}}
\newcommand{\Hom}{\operatorname{Hom}}
\newcommand{\Id}{\operatorname{Id}}
\newcommand{\im}{\operatorname{Im}}
\newcommand{\Ker}{\operatorname{Ker}}
\newcommand{\map}{\operatorname{map}}
\newcommand{\umap}{\underline{\map}\sb{\ast}}
\newcommand{\mapa}{\map\sb{\ast}}
\newcommand{\Map}{\operatorname{Map}}
\newcommand{\MapE}{\Map\sb{\TE}}
\newcommand{\New}{\operatorname{New}}
\newcommand{\Obj}{\operatorname{Obj}}
\newcommand{\op}{\sp{\operatorname{op}}}
\newcommand{\PO}[1]{{\mathbf P}\sp{#1}}
\newcommand{\Po}[1]{\Pc\sp{#1}}
\newcommand{\PNK}[2]{\PO{#1}\sb{#2}}
\newcommand{\Pnk}[2]{\Po{#1}\sb{#2}}
\newcommand{\proj}{\operatorname{proj}}
\newcommand{\sh}[1]{\operatorname{sh}\sp{#1}}
\newcommand{\St}{\operatorname{St}}
\newcommand{\Tot}{\operatorname{Tot}}
\newcommand{\sk}{\operatorname{sk}}
\newcommand{\we}{\operatorname{w.e.}}
\newcommand{\Ext}{\operatorname{Ext}}
\newcommand{\diag}{\operatorname{diag}}
\newcommand{\adj}[2]{\substack{{#1}\\ \rightleftharpoons \\ {#2}}}
\newcommand{\Sig}[1]{\Sigma\sb{#1}}
\newcommand{\C}{{\mathcal C}}
\newcommand{\II}{{\mathcal I}}
\newcommand{\JJ}{{\mathcal J}}
\newcommand{\Le}{{\mathcal L}\sb{\bE}}
\newcommand{\Lel}[1]{\Le\sp{#1}}
\newcommand{\LE}{\Lel{\lambda}}
\newcommand{\Pc}{{\mathcal P}}
\newcommand{\Ree}{{\mathcal R}\sb{\bE}}
\newcommand{\Rel}[1]{\Ree\sp{#1}}
\newcommand{\RE}{\Rel{\lambda}}
\newcommand{\Se}{{\mathcal S}\sb{\bE}}
\newcommand{\Sel}[1]{\Se\sp{#1}}
\newcommand{\SE}{\Sel{\lambda}}
\newcommand{\Te}{{\mathcal T}\sb{\bE}}
\newcommand{\Tel}[1]{\Te\sp{#1}}
\newcommand{\cTE}{\Tel{\lambda}}
\newcommand{\sTE}{\mathscr{T}\sb{\bE}}
\newcommand{\Ve}{{\mathcal V}\sb{\bE}}
\newcommand{\Vel}[1]{\Ve\sp{#1}}
\newcommand{\VE}{\Vel{\lambda}}
\newcommand{\G}{\mathcal{G}}
\newcommand{\Gp}{{\sf Gp}}
\newcommand{\Xill}[1]{\Xi\sb{#1}}
\newcommand{\Xil}{\Xill{\lambda}}
\newcommand{\SGa}{\Xil}
\newcommand{\Set}{\mbox{\sf Set}}
\newcommand{\Seta}{\Set\sb{\ast}}
\newcommand{\Sp}{{\sf Sp}}
\newcommand{\ST}{\Sp\sp{\TE}}
\newcommand{\STn}[1]{(\ST)\sp{#1}\sb{0}}
\newcommand{\Spaces}{{\sf Spaces}}
\newcommand{\Spa}{\Spaces\sb{\ast}}
\newcommand{\eSp}{\widehat{\Hom}}
\newcommand{\wSp}{\widehat{\Sp}}
%
%
\newcommand{\ff}{\mathfrak{f}}
\newcommand{\fg}{\mathfrak{g}}
\newcommand{\fM}{\mathfrak{M}}
\newcommand{\ME}{\fM\sb{E}}
\newcommand{\MES}{\ME\sp{\St}}
\newcommand{\MESl}{\ME\sp{\St,\lambda}}
\newcommand{\fmE}{\fM\sb{\bE}}
\newcommand{\fmel}[1]{\fmE\sp{#1}}
\newcommand{\fME}{\fmel{\lambda}}
\newcommand{\fW}{\mathfrak{W}}
\newcommand{\fU}{\mathfrak{U}}
\newcommand{\fWd}{\fW\sb{\bullet}}
\newcommand{\fUd}{\fU\sb{\bullet}}
\newcommand{\fX}{\mathfrak{X}}
\newcommand{\hfX}{\hat{\fX}}
\newcommand{\fXd}{\fX\sb{\bullet}}
\newcommand{\fY}{\mathfrak{Y}}
\newcommand{\fYd}{\fY\sb{\bullet}}
\newcommand{\fZ}{\mathfrak{Z}}
\newcommand{\The}{\mathbf{\Theta}\sb{E}}
\newcommand{\Thw}{\mathbf{\Theta}\sb{W}}
\newcommand{\Thel}[1]{\The\sp{#1}}
\newcommand{\TE}{\Thel{\lambda}}

%
\newcommand{\cL}{\mathcal{L}}
%
\newcommand{\FF}{\mathbb F}
\newcommand{\Fp}{\FF\sb{p}}
\newcommand{\HFp}{\bH\Fp}
\newcommand{\NN}{\mathbb N}
\newcommand{\QQ}{\mathbb Q}
\newcommand{\ZZ}{\mathbb Z}
%
%
\newcommand{\bA}{{\mathbf A}}
\newcommand{\bB}{{\mathbf  B}}
\newcommand{\bD}{{\mathbf  D}}
\newcommand{\bE}{{\mathbf  E}}
\newcommand{\bF}{\mathbf{F}}
\newcommand{\Es}{E\sp{\ast}}
\newcommand{\bH}{{\mathbf H}}
\newcommand{\bK}{\mathbf{K}}
\newcommand{\bL}{\mathbf{L}}
\newcommand{\bM}{{\mathbf M}}
\newcommand{\hM}{\widehat{\bM}}
\newcommand{\bN}{{\mathbf N}}
\newcommand{\hN}{\widehat{\bN}}
\newcommand{\bP}{{\mathbf P}}
\newcommand{\hP}{\widehat{\bP}}
\newcommand{\bR}{{\mathbf R}}
\newcommand{\bS}[1]{{\mathbf S}\sp{#1}}
\newcommand{\bU}{{\mathbf U}}
\newcommand{\wV}{\widetilde{V}}
\newcommand{\Vd}{\wV\sb{\bullet}}
\newcommand{\bW}{{\mathbf W}}
\newcommand{\Wu}{\bW\sp{\bullet}}
\newcommand{\Uu}{\bU\sp{\bullet}}
\newcommand{\Du}{\Delta\sp{\bullet}}
\newcommand{\hWu}{\widehat{\bW}\sp\bullet}
\newcommand{\bX}{{\mathbf X}}
\newcommand{\Xu}{\bX\sp{\bullet}}
\newcommand{\bY}{{\mathbf Y}}
\newcommand{\bZ}{{\mathbf Z}}
\newcommand{\bQ}{{\mathbf Q}}
\newcommand{\bo}{\mathbf 1}

\newcommand{\Gu}{G^\bullet}
%
%
\newcommand{\Eut}[2][ ]{$E\sp{#2}$-term{#1}}
\newcommand{\Elt}[2][ ]{$E\sb{#2}$-term{#1}}
\newcommand{\Ett}[1][ ]{$E\sb{2}$-term{#1}}

\newcommand{\eM}{e \sb{\bM}}
\newcommand{\peM}{e \sb{\bM'}}
\newcommand{\abM}{b \sb{\bM}}
\newcommand{\pbM}{b \sb{\bM'}}


\newcommand{\wlY}{\widehat{\lambda \sb{\bY}}}
\newcommand{\wlZ}{\widehat{\lambda \sb{\bZ}}}
\newcommand{\card}{\operatorname{card}}
\newcommand{\sJ}{\mathscr{J}}

\setcounter{section}{-1}
\title{Mapping algebras and the Adams spectral sequence}

\author{David Blanc}
\email{blanc@math.haifa.ac.il}
%
\author{Surojit Ghosh}
\email{surojitghosh89@gmail.com; gsurojit@campus.haifa.ac.il}
\address{Department of Mathematics, University of Haifa, Haifa-3498838, Israel.}
\maketitle

\begin{abstract}
For a suitable ring spectrum, such as \w[,]{\bE=\HFp} the \Ett of the $\bE$-based
Adams spectral sequence for a spectrum $\bY$ may be described in terms of its cohomology
\w[,]{\Es\bY} together with the action of the primary operations \w{\Es\bE} on it.
We show how the higher terms of the spectral sequence can be similarly described
in terms of the higher order truncated $\bE$-\emph{mapping algebra} for $\bY$ \wh
that is truncations of the function spectra \w{\F(\bY,\bM)} for various $\bE$-modules
$\bM$, equipped with the action of \w{\F(\bM,\bM')} on them.
\end{abstract}

\maketitle
%
%
\sect{Introduction}

The Adams spectral sequence is an important tool in stable homotopy theory, originally
introduced in \cite{Ada58} in order to compute the stable homotopy groups
of the sphere (at a prime $p$), using the Eilenberg-MacLane spectrum \w[.]{\bE=\HFp}
It was later generalized by Novikov in \cite{Nov67} to more general ring spectra
$\bE$.

The information needed to determine the \Ett of the $\bE$-based Adams spectral sequence
for a spectrum $\bY$ are the $\bE$-cohomology groups of $\bY$, together with the action of
the primary $\bE$-cohomology operations on $\bY$. More generally, we must consider
the homotopy classes \w{[\bY,\bM]} for all $\bE$-module spectra $\bM$, together with the
action of \w{[\bM,\bM']} on them (see \cite[3.1]{Bou01} to understand why this may
be necessary for general $\bE$).

However, it is not  \emph{a priori} clear what higher order information is
needed in order to determine the \Elt[s]{r} for \w[.]{r>2} As we shall see,
it turns out that it is sufficient to know the \wwb{r-2}truncation
\w{\PO{r-2}\ME\bY\lra{0}} (see \S \ref{window}) of the $\bE$-\emph{mapping algebra}
\w{\ME\bY} for $\bY$ \wh that is the function spectra
\w{\F(\bY,\bM)} for various $\bE$-modules $\bM$, equipped with the action of
\w{\F(\bM,\bM')} on them.

An explicit computation of was carried

Work of the late Hans Baues and his collaborators shows that the \Elt{3} of the usual
Adams spectral sequence, for \w{\bY=\bS{0}} and for \w[,]{\bE=\HFp} 
might be accessible to computation using the ``secondary Steenrod algebra'',
equivalent to the first Postnikov section \w{P\sp{1}\ME\bE} of the \ww{\Fp}-mapping
algebra (see \cite{BJ11}).
The structure of the analogous unstable Adams spectral sequence was studied in
\cite{BBC19} (which identifies the \Elt[s]{r} as certain truncated derived functors)
and in \cite{BBS17} (which describes the differentials as higher cohomology operations).

Following \cite{BS18}, we use a specific version of \w{\ME\bY} to
construct a cosimplicial Adams resolution \w[,]{\bY\to\Wu} so that the
homotopy spectral sequence for \w{\F(\bZ,\Wu)} is the $\bE$-based Adams spectral sequence
for \w[.]{\F(\bZ,\bY)} Analysis of the differential \w{d\sb{r-1}} shows that it
only depends on the \wwb{r-2}truncation of \w[,]{\ME\Wu} and thus
that the \Elt[s]{r} are determined by \w[.]{\PO{r-2}\ME\bY\lra{0}}

\begin{mysubsection}{Outline}
Section \ref{css} recalls some facts about the category \w{\Sp} of
symmetric spectra and Section \ref{cma} defines our main technical
tool: spectral functors defined on small categories \w{\TE} of
$\bE$-modules, for a fixed ring spectrum $\bE$, and their
truncations. In Section \ref{cdsc} we define \emph{mapping
algebras} \wh a generalization of the representable spectral
functor \w{\fME\bY} (defined by \w[).]{\bM\mapsto\F(\bY,\bM)} We
use this to construct a monad on spectra, which we analyze in
Section \ref{csma} in order to overcome certain set-theoretical
difficulties. This allows to obtain our first result, in Section
\ref{ccsass}:

\begin{thma}
If $\bE$ is a ring spectrum and $\bY$ an \ww{\bE}-good symmetric spectrum, we
can associate to the representable mapping algebra \w{\fME\bY} a cosimplicial
spectrum \w{\Wu} such that \w{\Tot \Wu} is $\bE$-equivalent to $\bY$.
\end{thma}
\noindent See Theorem \ref{main1} below\vsm .

In Section \ref{cdiff}  we analyze the differentials in
the $\bE$-based Adams spectral sequence for \w{\F(\bZ,\bY)} (in its cosimplicial
version), and show:

\begin{thmb}
Given $\bE$, $\bZ$, and $\bY$ as above, for each \w[,]{r\geq 1} the
\ww{d\sb{r}}-differential in the $\bE$-based Adams spectral sequence for
\w[,]{\F(\bZ,\bY)} and thus its \ww{E\sb{r+1}}-term, can be calculated from
the cosimplicial \wwb{r-1}truncated space \w[.]{\Pnk{r-1}{0}\fME\bZ\lin{\Wu}}
\end{thmb}
\noindent See Theorem \ref{pdiff} below\vsm.

We then use the resolution model category of truncated spectral functors to deduce:

\begin{thmc}
If \w{\bE=\bH R} for a commutative ring $R$, $\bZ$ is a fixed finite spectrum, and
$\bY$ is a $\bE$-good spectrum, then for any \w{r\geq 0} the
\Elt{r+2} of the $\bE$-based Adams spectral sequence for \w{\F(\bZ,\bY)} is
determined by the $r$-truncation \w{\Pnk{r}{0}\fME\bY} of the $\bE$-mapping
algebra of $\bY$.
\end{thmc}
\noindent See Theorem \ref{main2} below.
\end{mysubsection}

\begin{notn}
  We denote the category of sets by \w[,]{\Set} that of pointed sets by \w[,]{\Seta}
that of simplicial sets (called \emph{spaces}) by \w[,]{\Spaces}  and
  that of pointed simplicial sets by \w[.]{\Spa} For \w[,]{X,Y\in\Spa}
  \w{X\wedge Y:=(X\times Y)/(X\vee Y)} is the usual smash product.
    For any category $\C$ and \w[,]{A,B\in\C} we write \w{\C(A,B)} for
  \w[.]{\Hom\sb{\C}(A,B)\in\Set}
\end{notn}

\begin{ack}
We would like to thank the referee for his or her detailed and pertinent comments.
This research was supported by Israel Science Foundation grant 770/16.
\end{ack}

%
%
\sect{Symmetric spectra}
\label{css}

In this section we recall from \cite{HSS} some basic facts about symmetric spectra.
We prefer this model for the stable homotopy category because it has useful set-theoretic
properties.

\begin{defn}\label{dsymsp}
A \emph{symmetric spectrum} is a sequence of pointed spaces (simplicial sets)
\w{\bX=(X_n)_{n \geq 0}} equipped with:
\begin{enumerate}
\renewcommand{\labelenumi}{(\arabic{enumi})}
\item A pointed map \w{\sigma :S\sp{1}\wedge X_n \to X_{n+1}} for each \w[;]{n \geq 0}
\item A basepoint-preserving left action of the symmetric group \w{\Sig{n}} on \w[,]{X_n}
such that the composite
$$
\sigma^p = \sigma \circ (S\sp{1}\wedge\sigma) \circ \cdots
\circ (S\sp{p-1}\wedge \sigma):S^p \wedge X_n \to X_{n+p}
$$
\noindent is \ww{\Sig{p}\times\Sig{n}}-equivariant for \w{p \geq 1} and
\w[.]{n \geq 0}
\end{enumerate}

A map \w{f:\bX \to\bY} of symmetric spectra is a sequence of \ww{\Sig{n}}-equivariant
maps \w{f_n : X_n \to Y_n} such that the diagram
$$
\xymatrix{S^1 \wedge X_n \ar[d]_{\Id\sb{S^1} \wedge f_n}\ar[r]^{\sigma}   & X_{n+1}\ar[d]^{f_{n+1}}\\
  {S^1} \wedge Y_n \ar[r]^{\sigma}   & Y_{n+1}}
$$
\noindent commutes for all \w[.]{n \geq 0} We denote the category of symmetric spectra by
\w[.]{\Sp}

The \emph{smash product} of \w{\bX,\bY\in\Sp} is defined to be the symmetric spectrum
\w{\bX\otimes\bY} given by
$$
(\bX\otimes\bY)_n: = \bigvee_{p+q=n} {\Sig{n}}_+ \wedge_{\Sig{p}\times\Sig{q}} X_p \wedge Y_q.
$$
Given \w{\bX\in\Sp} and a pointed space $K$, the symmetric spectrum
\w{K\otimes \bX} is defined by \w{(K\otimes \bX)\sb{n}:=K\wedge X\sb{n}}
(see \cite[\S 1.3]{HSS}).
\end{defn}

\begin{mysubsection}{The model category of symmetric spectra}
The (stable) model structure on \w{\Sp} is defined in \cite{HSS} as follows:

A map \w{f:\bX\to\bY} of symmetric spectra is
\begin{enumerate}
\renewcommand{\labelenumi}{(\roman{enumi})}
\item a \emph{stable equivalence} if it induces an
isomorphism in stable homotopy groups (forgetting the \ww{\Sigma_n}-actions).
\item a \emph{level trivial fibration} if at each level it is
a trivial Kan fibration of simplicial sets.
\item a \emph{stable cofibration} if it has the left lifting property with
respect to level trivial fibrations.
\item a \emph{stable fibration} if it has the right lifting property with respect to
  every stable cofibration which is a stable equivalence.
\end{enumerate}

By \cite[Theorem 3.4.4]{HSS}, the classes of stable equivalences, cofibrations,
and fibrations define a proper, simplicial, symmetric model category structure on \w[,]{\Sp}
monoidal with respect to $\otimes$. The simplicial enrichment is given by
$$
\map\sb{\Sp}(\bX,\bY)\sb{n} = \Sp(\Delta[n]\sb{+}\otimes\bX,\bY)~.
$$
\end{mysubsection}

\begin{defn}\label{function_sp}
Given \w[,]{\bX,\bY\in\Sp} the \emph{function spectrum} \w{\F(\bX,\bY)} is defined to
be the symmetric spectrum given by
\begin{myeq}\label{eqfunsp}
\F(\bX,\bY)\sb{n} := \map\sb{\Sp}(\bX,\sh{n}\bY)
\end{myeq}
where \w{\sh{n}\bY} is the $n$-shifted symmetric spectrum given by \w[.]{(\sh{n}\bY)_k= Y_{n+k}}
The action of the symmetric group \w{\Sigma_n} is induced from the action on \w[.]{\sh{n}\bY} One may see \cite[Remark 2.2.12]{HSS} for the symmetric spectra structures on \w{\sh{n}\bY} and \w[.]{\F(\bX,\bY)}
\end{defn}

\begin{remark}
We have adjoint functors
\w{\Spa\adj{\Sigma\sp{\infty}}{\Omega\sp{\infty}} \Sp} with
$$
\Omega\sp{\infty}\F(\bX,\bY)~\simeq~\map_{\Sp}(\bX,\bY)~.
$$
Moreover, given \w[,]{\bX,\bY,\bZ\in\Sp} by \cite[Theorem 2.2.10]{HSS} there is a
natural adjunction isomorphism
\begin{myeq}\label{eqfunmap}
\Sp(\bX \otimes\bY,\bZ) \cong \Sp(\bX, \F(\bY,\bZ))~.
\end{myeq}

For any \w[,]{\bX\in\Sp} the function spectra \w{\Omega\bX:=\F(\bS{1},\bX)} and
  \w{P\bX:=\F(\Sigma\sp{\infty}\Delta[1]\sb{+},\bX)} are called \emph{loop} and \emph{path}
  spectra of $\bX$, respectively. Note that \w{\F(\bX,-)} commutes with the loop and path
  constructions.
\end{remark}

\begin{defn}\label{ringmod}
A \emph{symmetric ring spectrum} is a symmetric spectrum $\bR$ together with
spectrum maps \w{m:\bR\otimes\bR\to\bR} (\emph{multiplication}) and
\w{\iota:\bS{0}\to\bR} (the \emph{unit} map) with
\w[,]{m \circ (m \otimes \Id) = m \circ (\Id \otimes m)} such that
\w{m\circ(\iota\otimes\Id):\bS{0}\otimes\bR\to\bR}
and \w{m\circ(\Id\otimes\iota):\bR\otimes\bS{0}\to\bR} are the standard equivalences.

An \emph{$\bR$-module} for a symmetric ring spectrum $\bR$ is a symmetric spectrum $\bM$
equipped with a spectrum map \w{\mu:\bR\otimes\bM\to\bM} with
\w{\mu\circ(m\otimes\Id)=\mu\circ(\Id\otimes \mu)} and the map \w{\mu \circ(\iota \otimes \Id): \bS{0}\otimes \bM \to \bM} is the standard equivalence.

For any symmetric spectrum $\bY$ and symmetric ring spectrum $\bR$, the function spectrum
\w{\F(\bY,\bR)} admits a module structure over $\bR$.
\end{defn}

\begin{notn}
For any symmetric spectrum $\bX$, we write \w{\|\bX\|:=\sup\sb{n\in\NN}\|X_n\|}
(the cardinality of the simplicial set). Note that if $\lambda$ is a limit cardinal and
\w[,]{\|\bX\|<\lambda} then $\bX$ is \emph{$\lambda$-small}
in the usual sense (see \cite[Definition 10.4.1]{Hir03}). 
\end{notn}

%
%
\sect{Spectral functors}
\label{cma}

Our main technical tool in this paper is the following:

\begin{defn}\label{mapalg}
  Given a symmetric ring spectrum $\bE$, let \w{\ModE} denote (a skeleton of) the full
  subcategory of $\bE$-module spectra in \w[.]{\Sp} We consider spectral functors
  \w{\fX:\ModE\to\Sp}
(that is, functors respecting the spectral enrichment), and write \w{\fX\lin{\bM}\in\Sp}
  for the value of $\fX$ at \w[.]{\bM\in\ModE} The spectral enrichment \wh or
  rather, its truncations \wh will play a central role in the paper; our main point
  is that these provide the data needed to compute the higher differentials in the
  Adams spectral sequence.
A functor of the form \w{\bM\mapsto\F(\bY,\bM)} for some fixed \w{\bY\in\Sp}
will be called \emph{representable}.

If $\lambda$ is some limit cardinal, the corresponding $\bE$-\emph{spectral theory}
is the full subcategory \w{\TE} of \w{\ModE} consisting of
all $\bE$-module spectra which are $\lambda$-small.
We denote by \w{\ST} the category of all spectral functors from \w[.]{\TE}
Note that the \ww{\TE}-spectral functor represented by \w[,]{\bY\in\Sp}
denoted by \w[,]{\fME\bY} is a homotopy functor (that is, it preserves weak
equivalences). When \w[,]{\bY\in\TE} we say that \w{\fME\bY} is \emph{free}.
Observe that \w{\fME\bY} is contravariant in the variable $\bY$.
\end{defn}

\begin{lemma}\label{yoneda}
If $\fX$ is any \ww{\TE}-spectral functor and \w{\fME\bM} is free (for
\w[),]{\bM\in\TE} there is a natural isomorphism
\w[.]{\ST(\fME\bM,\,\fX)~\cong~\Sp(\bS{0},\,\fX\lin{\bM})}
In particular, if \w{\fX=\fME\bY} is representable, a map \w{\bS{0}\to\fX\lin{\bM}}
corresponds to a map of spectra \w{\bY\to\bM} by \wref{eqfunmap} \wwh that is,
$$
\ST(\fME\bM,\, \fME\bY)~\cong~\Sp(\bS{0}, \fME\bY\lin{\bM})~\cong~\Sp(\bY, \bM)~.
$$
\end{lemma}

\begin{proof}
This follows from the enriched Yoneda Lemma (see \cite[Proposition 2.4]{Kel82}).
\end{proof}

\begin{remark}\label{rhtpyfun}
Any \ww{\TE}-spectral functor $\fX$ which is a  homotopy functor preserves homotopy
  pullbacks and pushouts (which are equivalent in \w[),]{\Sp} by
  \cite[Proposition 4.1]{Cho16}. In particular, it preserves the path, loop, and
  suspension operations on spectra, up to weak equivalence. Thus
\begin{myeq}\label{eqomegapi}
  \pi\sb{k}\fX\lin{\bM}~\cong~\pi\sb{0}\Omega\sp{k}\fX\lin{\bM}~\cong~
  \pi\sb{0}\fX\lin{\Omega\sp{k}\bM}~.
\end{myeq}
\noindent for any \w[.]{\bM\in\TE}
\end{remark}

\begin{mysubsection}{Truncation of spectral functors}
\label{window}
For each \w[,]{n\in\ZZ} consider the Postnikov section functor
\w{\PO{n}:\Sp \to \Sp} (localization with respect to \w[),]{\bS{n+1}\to\ast}
killing all homotopy groups in dimensions \w[,]{>n} and the \wwb{k-1}connected cover
\w{\lra{k}:\Sp \to \Sp} (colocalization with respect to \w{\ast\to \bS{k+1}}
\wwh see \cite[1.2 \& 5.1]{Hir03}).
When \w[,]{n\geq k} write \w{\PNK{n}{k}} for the composite \w[.]{\PO{n}\circ\lra{k}}

Note that for any \w{\bX,\bY\in\Sp} we have \w{\F(\bX,\bY)\sb{0}=\map\sb{\Sp}(\bX,\bY)}
(the simplicial enrichment), by \wref[.]{eqfunsp}
For any \w[,]{n\geq 0} we may define \w{\Pnk{n}{0}} in
\w{\Spaces} or \w{\Spa} by composing the \wwb{n+1}coskeleton functor with a
functorial fibrant replacement commuting with products, so it is monoidal in \w{\Spa}
with respect to cartesian products (see \cite[9.1.14]{Hir03}), with
\w[.]{[\PNK{n}{0}\F(\bX,\bY)]\sb{0}\simeq\Pnk{n}{0}\map\sb{\Sp}(\bX,\bY)}

We can therefore define a new enrichment on \w{\Sp} in \w{(\Spa,\times)} by
$$
\umap(\bX,\bY)~:=~\Pnk{n}{0}\F(\bX,\bY)\sb{0}~,
$$
and call any \w{\fX:\TE\to\Sp} respecting this enrichment a \emph{truncated}
\ww{\TE}-spectral functor, and their category will be denoted by \w[.]{\STn{n}}
In particular, applying \w{\Pnk{n}{0}} to any
spectral functor $\fX$ yields such a functor, defined by
\w[,]{(\Pnk{n}{0}\fX)\lin{\bM}:=\Pnk{n}{0}((\fX\lin{\bM})\sb{0})} which we call
simply the \emph{$n$-truncation} of $\fX$:
explicitly, the action of \w{\TE} on $\fX$, in the form of maps
of spectra
$$
\fX\lin{\bM}\wedge\F(\bM,\bN)\to\fX\lin{\bN}~,
$$
\noindent yields a map of simplicial sets
\w[,]{(\fX\lin{\bM})\sb{0}\wedge\F(\bM,\bN)\sb{0}\to(\fX\lin{\bN})\sb{0}}
and by precomposing with the quotient map
\w{(\fX\lin{\bM})\sb{0}\times\F(\bM,\bN)\sb{0}\to
  (\fX\lin{\bM})\sb{0}\wedge\F(\bM,\bN)\sb{0}} and applying our monoidal \w{\Pnk{n}{0}}
we obtain an action of \w{\TE} with its new enrichment:
\begin{myeq}\label{eqaction}
(\Pnk{n}{0}\fX)\lin{\bM}\times\umap(\bM,\bN)\to(\Pnk{n}{0}\fX)\lin{\bN}~.
\end{myeq}

Note that \w{\Pnk{n}{0}\fX} is not itself a spectral functor in the sense
of \S\ref{mapalg}. However, we still have the same notion of weak equivalences
in \w{\STn{n}} \wwh namely, natural transformations inducing weak equivalences for each
\w[.]{\bM\in\TE}

For \w[,]{\bM\in\TE} we say that \w{\Pnk{n}{0}\fME\bM} is a \emph{free} truncated
\ww{\TE}-spectral functor, since we have the following analogue of Lemma \ref{yoneda}:
\end{mysubsection}

\begin{lemma}\label{lyoneda}
If \w{\fX\in\STn{n}} and \w[,]{\bM\in\TE} there is a
natural isomorphism
$$
\STn{n}(\Pnk{n}{0}\fME\bM,\,\fX)~\cong~\Sp(\bS{0},\,\fX\lin{\bM})~.
$$
\end{lemma}

\begin{remark}
Since any \ww{\TE}-spectral homotopy functor $\fX$ commutes up to weak equivalence
with $\Omega$, we have
\begin{myeq}\label{eqtruncwe}
\Pnk{n}{0}\fX\lin{\Omega\sp{k}\bM}~\stackrel{\we}{\simeq}~
\Pnk{n}{0}\Omega\sp{k}\fX\lin{\bM}~\stackrel{\we}{\simeq}~
\Omega\sp{k}(\PNK{n+k}{k}\fX)\lin{\bM})~.
\end{myeq}
\noindent Therefore, \w{\Pnk{n}{0}\fX} determines \w{\PNK{n+k}{k}\fX}
up to homotopy for all \w[.]{k\in\ZZ}
\end{remark}

\begin{mysubsection}{Model category structures}
\label{smcs}
By \cite{HSS}, \w{\Sp} has a proper simplicial model category structure, and
by \cite{MMSS01}, it is cofibrantly generated.
Since \w{\TE} is small, by \cite[Theorems 11.1.6 \& 13.1.14]{Hir03} there is
a projective proper simplicial model category structure on \w[,]{\ST} in which
the weak equivalences and fibrations are level-wise \wh in particular, a
map\w{\ff:\fX \to \fX'} in \w{\ST} is a weak equivalence if and only
if \w{\ff\sb{\ast}:\fX\lin{\bM} \to \fX'\lin{\bM}} is a weak equivalence in
\w{\Sp} for each \w[.]{\bM \in \TE}

We may similarly define a \ww{\Pnk{n}{0}}-\emph{weak equivalence} of spectral functors
to be a map \w{\ff:\fX\to\fY} inducing a weak equivalence after applying
\w[.]{\Pnk{n}{0}} Of course, for homotopy spectral functors these are the same those
just defined, by \wref[,]{eqtruncwe} but in general they are different. By applying
Bousfield (co)localization to the above we obtain the
\ww{\Pnk{n}{0}}-\emph{model structure on} \w{\ST} (see \cite[Ch.\ 3]{Hir03}).
\end{mysubsection}

\begin{prop}\label{pmodelcat}
The \ww{\Pnk{n}{0}}-model category structure on \w{\ST} is right proper.
\end{prop}

\begin{proof}
The Postnikov section functor \w{\Po{n}} is a nullification, so a left Bousfield
localization. Hence, by \cite[Proposition 3.4.4]{Hir03} we have a left proper model
structure on the image of \w{\Po{n}} in \w[.]{\ST} The argument of
\cite[Theorem 9.9]{Bou01} (which also works in \w[)]{\Sp} shows that it is also
right proper. Since taking connected covers is a right Bousfield localization, by
\cite[\textit{loc.\ cit.}]{Hir03} we see that \w{\STn{n}} is right proper.
\end{proof}

\begin{mysubsection}{Homotopy groups}
  \label{shg}
The homotopy groups \w{\pi\sb{i}\fX\lin{\bM}} are used to define weak
  equivalences for a spectral functor $\fX$, and we will need to identify the minimal
  information needed to determine them. In fact, by \wref{eqomegapi} we need only the
  $0$-th (stable) homotopy group, if $\fX$ is a homotopy functor.

Since any spectrum $\bB$ is a homotopy group object, with group operation
\w{\mu:\bB \times\bB\to\bB} and inverse \w[,]{\nu:\bB \to\bB}
for any \w{\bA\in\Sp} we have
\w{\mu\sb{\ast}: \Sp(\bA, \bB) \times \Sp(\bA,\bB) \to \Sp(\bA,\bB)}
and \w[.]{\nu\sb{\ast}:\Sp(\bA,\bB) \to\Sp(\bA,\bB)} As by \wref[,]{eqfunmap}
\w[,]{\Sp(\bA,\bB)=\Sp(\bS{0},\F(\bA,\bB))} we may define a relation
 $\sim$ on \w{\Sp(\bA,\bB)} by \w{f \sim g} if and only if there
  exists \w{F\in\Sp(\bS{0},P\F(\bA,\bB))} such that \w[,]{\mu_*(\nu _*(g), f) = p\sb{\ast}F}
where \w{p:P\bX\to\bX} is the path fibration. We then see:
\end{mysubsection}

\begin{lemma}\label{lerel}
If \w{\bA\in\Sp} is cofibrant and \w{\bB\in\Sp} is fibrant, the relation $\sim$ is
an equivalence relation on \w{\Sp(\bA,\bB)} which coincides with
the (left or right) homotopy relation on \w[,]{\Sp(\bA,\bB)} which we denote
by $\simeq$.
\end{lemma}

As usual, we write \w{[\bA,\bB]} for \w[.]{\Sp(\bA,\bB)/\sim}

\begin{proof}
The fact that $\sim$ is an equivalence relation is readily verified.
Given two homotopic maps \w[,]{f \simeq g:\bA\to\bB} \w{\mu_*(\nu_*(g), f)} is
nullhomotopic, so there is \w{F:\bS{0}\to P\F(\bA,\bB)} with
\w[.]{\mu_*(\nu_*(g), f) = p\sb{\ast}F}
Conversely, given \w{F:\bS{0}\to P\F(\bA,\bB))} with
\w[,]{\mu_*(\nu_*(g), f) =p\sb{\ast}F} we see that \w{\mu_*(\nu_*(g), f)} is
nullhomotopic, so
\w[.]{g \simeq \mu_*(g,\ast)\simeq\mu_*(g, \mu_*(\nu_*(g), f))
  \simeq\mu_*(\mu_*(g , \nu_*(g)), f)\simeq\mu_*(\ast, f)\simeq f}
\end{proof}

\begin{remark}\label{rarrowset}
If we let \w{\bo=(0\to 1)} denote the one-arrow category, with a
single non-identity map, and \w{\Seta\sp{\bo}} the corresponding
functor category into pointed sets, we may define a functor
\w{\widehat{\rho}:\Sp\to\Seta\sp{\bo}}
by \w[,]{\bX\mapsto [\Sp(\bS{0},P\bX)\xra{p}\Sp(\bS{0},\bX)]} and deduce:
\end{remark}

\begin{cor}\label{grpstr}
For fibrant \w{\bB\in\Sp} the functor \w{\pi\sb{0}\F(-,\bB):\Sp\sb{\cof}\to\Gp} (on the
subcategory of cofibrant spectra) factors through \w[.]{\widehat{\rho}\circ\F(-,\bB)}
\end{cor}

%
%
\sect{Mapping algebras}
\label{cdsc}

We now show that \ww{\TE}-spectral functors $\fX$ having a certain property
(called \emph{mapping algebras}) are representable, up to weak equivalence.
To do so, in Section \ref{ccsass} we will construct a cosimplicial spectrum \w{\Wu}
using this structure, and show that \w{\Tot\Wu} realizes $\fX$, up to weak equivalence.

The discussion in \S \ref{shg} suggests the following:

\begin{mysubsection}{The arrow set category}\label{arrcat}
  For a fixed limit cardinal $\lambda$, with \w{\TE} as above, let \w{\Gamma\TE} denote
  the directed graph associated to the underlying category of \w{\TE}
  (see \cite{HasseG}). We then define an
  \emph{arrow set} $A$ to be a function \w{A:\Gamma\TE\to \Seta\sp{\bo}}
  (see \S \ref{rarrowset}) which assigns to each $\bM$ in \w{\TE}
  a map of pointed sets \w[,]{A(\chi \sb{\bM}):A(e\sb{\bM})\to A(b\sb{\bM})}
fitting into a commutative square
\mydiagram[\label{arrdiag}]{
A(e\sb{\bM})\ar[d]_{A(\chi \sb{\bM})} \ar[rr]^{A(e_j)} && A(e\sb{\bM'})\ar[d]^{A(\chi \sb{\bM'})}\\
    A(b\sb{\bM}) \ar[rr]^{A(b_j)} && A(b\sb{\bM'})
}
\noindent for each map \w{j: \bM \to \bM'} in \w[.]{\TE}

This is equivalent to having a functor from the free category on \w{\Gamma\TE}
to  \w[.]{\Seta\sp{\bo}} We denote the category of such arrow sets by \w[.]{\Xil}

For each fixed limit cardinal $\lambda$ we have a functor
\w[,]{\rho:\ST\to\Xil} where the arrow set \w{\rho(\fX)}
assigns to each map \w{j:\bM\to\bM'} in \w{\TE} the commutative square:
\mydiagram[\label{rho}]{
\Sp(\bS{0}, P\fX\lin{\bM})\ar[rr]^{p_*} \ar[d]^{P\fX\lin{j}} &&
  \Sp(\bS{0},\fX\lin{\bM})\ar[d]^{\fX\lin{j}} \\
\Sp(\bS{0}, P\fX\lin{\bM'})\ar[rr]^{p_*} && \Sp(\bS{0},\fX\lin{\bM'})
  }
The map \w{p_*} is induced by the path fibration \w{p\sb{\bY}:P\bY\to\bY}
for \w{\bY=\fX\lin{\bM}} (compare \S \ref{rarrowset}).
\end{mysubsection}

\begin{mysubsection}{Maps of arrow sets}\label{smaparr}
  Using \wref[,]{arrdiag} any \w{A, B \in \Xil} induce the following diagram:
\myrdiag[\label{mapsarrows}]{
  \Seta(A(e\sb{\bM}), B(e\sb{\bM}))\ar[d]^{B(\chi\sb{\bM})\sb{\ast}}
  \ar[rr]^{B(e\sb{j})\sb{\ast}}   &&
  \Seta(A(e\sb{\bM}), B(e\sb{\bM'})) \ar[d]^{B(\chi\sb{\bM'})_\ast}   &&
\Seta(A(e\sb{\bM'}), B(e\sb{\bM'})) \ar[ll]_{A(e\sb{j})\sp{\ast}} \ar[d]^{B(\chi\sb{\bM'})_\ast} \\
  \Seta(A(e\sb{\bM}), B(b\sb{\bM})) \ar[rr]^{B(b\sb{j})\sb{\ast}} &&
  \Seta(A(e\sb{\bM}), B(b\sb{\bM'})) &&
  \Seta(A(e\sb{\bM'}), B(b\sb{\bM'})) \ar[ll]_{A(e_j)\sp{\ast}} \\
  \Seta(A(b\sb{\bM}), B(b\sb{\bM})) \ar[u]_{A(\chi\sb{\bM})^\ast}
\ar[rr]^{B(b\sb{j})\sb{\ast}} && \Seta(A(b\sb{\bM}),B(b\sb{\bM'})) \ar[u]_{A(\chi\sb{\bM})^\ast} &&
  \Seta(A(b\sb{\bM'}), B(b\sb{\bM'}))\ar[u]_{A(\chi\sb{\bM'})^\ast} \ar[ll]_{A(b_j)^\ast}}
\noindent Thus \w{\Xil(A,B)} is a product over all maps \w{j: \bM \to \bM'} in \w{\TE}
of the limit of the diagrams \wref[.]{mapsarrows}
\end{mysubsection}

\begin{remark}\label{rarrowst}
  Our goal is to describe the minimal data needed to determine when a map of spectral
  functors \w{\ff:\fX\to\fY} is a weak equivalence (\S \ref{smcs}) \wh i.e.,
  assuming these are homotopy spectral functors,
  when \w{\ff\sb{\ast}:\pi\sb{0}\fX\lin{\bM}\to\pi\sb{0}\fY\lin{\bM}} is an
  isomorphism for all \w[.]{\bM\in\TE}

  By Corollary \ref{grpstr}, the map of arrow sets \w{\rho\ff:\rho \fX \to\rho \fY}
    suffices for this purpose: in fact, it is enough to consider its values only on
    the objects of \w{\TE} (i.e., the vertical arrows in \wref[.]{arrdiag}

The more complicated definition of arrow sets given above is necessary only for
the smallness argument in Section \ref{csma} below.  However, we do not require
that an arrow set be functorial with respect to the compositions in \w[,]{\TE} since
this is not needed for our purpose.
\end{remark}

\begin{notn}\label{cardA}
Let \w{\Xi:=\bigcup\sb{\lambda}\,\Xil} (the union taken over all limit cardinals).
This is a large category, which we need only
in order to be able to discuss all arrow sets at once.

In particular, for each arrow set \w[,]{A\in\Xi} let $\lambda$ be maximal such that
\w[,]{A\in\Xil} and write \w{\|A\|:=\sup\sb{\bM\in\TE}\{|A(\eM)|,\ |A(\abM)|\}}
(where \w{|B|} denotes the cardinality of a set $B$).
We write \w{\LE:\Sp\to\Xil\op} for \w[.]{\rho \circ\fME}
\end{notn}

\begin{mysubsection}{The Stover construction}
  \label{sstoverconst}
To describe the right adjoint \w{\RE:\Xil\op\to\Sp} to \w[,]{\LE} we recall a construction
  due to Stover (see \cite{Sto90} and compare \cite{BS18}):

  We want to have \w[.]{\Sp(\bY,\RE A) \cong \Xil(A,\LE\bY)} By the description of
  morphisms in \w{\Xil} (see \S \ref{arrcat}), it follows that the right hand side \wh
  that is, \w{\Xil(A,\LE\bY)} is the product over all \w{\bM \in \TE}
and \w{j :\bM \to \bM'} of the limit of the following diagram:

\myrdiag[\label{sa}]{
  \prod\limits_{A(\eM)}\Sp(\bY, P\bM)\ar[d]^{(p \sb{\bM})_\ast}
  \ar[rr]^{\prod\limits_{A(\eM)}(P_j)_\ast}   &&
  \prod\limits_{A(\eM)}\Sp(\bY, P\bM') \ar[d]^{(p \sb{\bM'})_\ast}   &&
  \prod\limits_{A(\peM)}\Sp(\bY, P\bM') \ar[ll]_{\top A(e_j)\sp{\ast}} \ar[d]^{(p \sb{\bM'})_\ast} \\
  \prod\limits_{A(\eM)}\Sp(\bY, \bM) \ar[rr]^{\prod\limits_{A(\eM)}(j)_\ast} &&
  \prod\limits_{A(\eM)}\Sp(\bY, \bM') &&
  \prod\limits_{A(\peM)}\Sp(\bY, \bM') \ar[ll]_{\top A(e_j)\sp{\ast}} \\
  \prod\limits_{A(\abM)}\Sp(\bY, \bM) \ar[u]_{A(\chi\sb{\bM})^\ast}
  \ar[rr]^{\prod\limits_{A(\abM)}(j)_\ast} && \prod\limits_{A(\abM)}\Sp(\bY, \bM')
  \ar[u]_{A(\chi\sb{\bM})^\ast} &&
  \prod\limits_{A(\pbM)}\Sp(\bY, \bM')\ar[u]_{A(\chi\sb{\bM'})^\ast} \ar[ll]_{A(b_j)^\ast}}

Note that \wref{sa} splits up as a product of smaller diagrams, indexed by a
single map \w{\phi:\bY\to\bM} in the left two slots of the bottom row. Moreover, this
diagram is really only relevant for nullhomotopic $\phi$.

Therefore, given \w{\ast\neq\phi \in A(\abM)} and \w[,]{j: \bM \to \bM'} we define
\w{\bQ^{(\bM,\phi, j)}} to be the limit of the following diagram:
\myrdiag[\label{dualstover}]{
  \prod\limits_{A(\chi \sb{\bM})^{-1}(\phi)}P\bM \ar[d]^{\prod p\sb{\bM}} \ar[rr]^{\prod Pj} &&
  \prod\limits_{A(\chi \sb{\bM})^{-1}(\phi)} P\bM' \ar[d]^{\prod p\sb{\bM'}}  &&
  \prod\limits_{A(\chi \sb{M'})^{-1}(A(b_j)(\phi))} P\bM'
  \ar[ll]_{\top A(e_j)\sb{\ast}}\ar[d]^{\prod p\sb{\bM'}} \\
  \prod\limits_{A(\chi \sb{\bM})^{-1}(\phi)} \bM \ar[rr]^{\prod j} &&
  \prod\limits_{A(\chi \sb{\bM})^{-1}(\phi)} \bM'  &&
  \prod\limits_{A(\chi \sb{\bM'})^{-1}(A(b_j)(\phi))} \bM' \ar[ll]_{\top A(e_j)\sb{\ast}}\\
      \bM \ar[u]_-{\diag}\ar[rr]^{j} && \bM' \ar[u]_-{\diag} &&
      \bM'\ar[u]_-{\diag}\ar@{=}[ll]}
\noindent Note that if \w[,]{\phi \in \im(A(\chi\sb{\bM}))} then
\w[.]{A(b_j)(\phi) \in \im(A(\chi\sb{\bM'}))} Thus if \w{A(b_j)(\phi)} is not in the
image of \w[,]{A(\chi \sb{M'})}
then all six products the in two top rows of \wref{dualstover} have empty indexing
sets, so \w[.]{\bQ^{(\bM,\phi, j)}=\bM}

Finally, in the special case where $\phi$ is actually the zero map
$\ast$, we set \w[.]{\bQ^{(\bM,\phi, j)}:=\prod\sb{A(\chi
\sb{\bM})^{-1}(\ast)}\Omega\bM}
\end{mysubsection}

\begin{defn}
For a fixed limit cardinal $\lambda$, a \emph{mapping algebra} is a spectral
  functor \w{\fX:\TE\to \Sp} preserving all limits of the form \wref{dualstover} in
  \w[.]{\TE} In particular, by an appropriate choice of arrow set, we see that
  such an $\fX$ preserves loops up to homotopy.
The category of all mapping algebras for $\lambda$ is
denoted by \w[.]{\MapE}

Note that any representable spectral functor \w{\fME\bY}
(see \S \ref{mapalg}) is necessarily a mapping algebra, since it preserves \emph{all}
limits in \w[,]{\ModE} and the diagram \wref{dualstover} is in fact in
\w{\TE\subset\ModE} (including the path fibrations $p$).
\end{defn}

From the discussion in \S \ref{sstoverconst} we conclude (as in
\cite[3.1.1 \& 4.1.1]{BS18}):

\begin{lemma}\label{lrightadj}
For a fixed limit cardinal $\lambda$, the right adjoint \w{\RE:\Xil\op\to\Sp} of
\w{\LE} is given on \w{A\in\Xil} by
\begin{myeq}\label{rightadj}
  \RE(A)~:=~\prod\limits_{\bM \in \TE} \prod\limits_{\phi \in A(b \sb{\bM})}
  \prod\limits_{j: \bM \to \bM'} \bQ^{(\bM,\phi, j)}~.
\end{myeq}
\end{lemma}

\begin{remark}\label{rxi}
The limits \w{\bQ^{(\bM,\phi, j)}} of \wref{dualstover} and the products of
  \wref{rightadj} always exist in \w[,]{\ModE} but they may or may not be
  in \w[.]{\TE} However, if we let \w{\Arr\TE} denote the set of all morphisms (between
  any two objects) in \w[,]{\TE} with cardinality \w[,]{|\Arr\TE|} and set
  \w{\kappa:=\max\{|\Arr\TE|,\|A\|\sp{\lambda}\}}
(see \S \ref{cardA}), we see that \w{\RE(A)} is in \w{\Thel{\nu(A)}}
for \w[,]{\nu(A):=\kappa\sp{\kappa}} say.
\end{remark}

\begin{remark}\label{rmonad}
  Since \w{\RE} is right adjoint to \w[,]{\LE} we obtain a monad
  \w{\cTE: = \RE\circ \LE:\Sp\to\Sp} with unit
  \w{\eta = \widehat{\Id_{\LE}}:\Id \to \cTE} and multiplication
  \w[,]{\mu = \RE\circ\widetilde{\Id_{\cTE}} : \cTE \circ \cTE \to \cTE}
  as well as a comonad \w{\SE: = \LE\circ \RE} on \w[,]{\Xil\op}  with counit
  \w{\epsilon : = \widetilde{\Id_{\RE}} : \SE  \to\Id} and comultiplication
  \w{\delta : = \LE
    \circ \widehat{ \Id_{\RE}} : \SE \to \SE \circ \SE} (see \cite[\S 8.6.1]{Wei94} for
  an explanation of the notation).
  \end{remark}

\begin{defn}
  A \emph{coalgebra} over the comonad \w{\SE} is an object \w{A \in \Xil\op} equipped
  with a section \w{\zeta\sb{A}: A \to \SE A} of the counit \w[,]{\epsilon: \SE A \to A}
with \w[.]{\SE\zeta \circ \zeta = \delta_A \circ \zeta}
\end{defn}

\begin{prop}\label{coalgebra}
Assume given a limit cardinal $\lambda$ and a \ww{\TE}-mapping algebra $\fX$
which extends to a \ww{\Thel{\kappa}}-mapping algebra for
\w[,]{\kappa=\nu(\rho\fME\RE\rho\fX))} in the notation of \S \ref{rxi}.
Then the corresponding arrow set \w{\rho \fX} has a
natural coalgebra structure over \w[.]{\SE}
\end{prop}

\begin{remark}\label{rcoalg}
The assumption clearly holds whenever $\fX$ is representable \wh but in this case
we already know that the arrow set \w{\LE \bY = \rho \fME\bY} has a coalgebra structure,
given by \w[.]{\zeta_{\LE \bY}=\LE(\eta)=\LE(\widehat{\Id_{\LE \bY}})}
\end{remark}

\begin{proof}
We want to construct a map \w{\zeta\sb{\rho \fX}} fitting into
  a commutative diagram
\mydiagram[\label{dualcoalg}]{
    \rho \fX \ar[d]_{\zeta_{\rho \fX}} \ar[rr]^{\zeta_{\rho \fX}} &&
  \SE(\rho \fX) \ar[d]^{\SE \zeta_{\rho \fX}}\\
 \SE(\rho \fX) \ar[rr]_{\zeta_{\SE(\rho \fX)}} &&
 \SE \SE(\rho \fX)}
 \noindent in \w{\Xi\op} (so all maps in $\Xi$ are in the opposite direction!)

 Since \w[,]{\SE(\rho \fX)= \rho \fME\RE (\rho \fX)} all objects in \wref{dualcoalg} are
 in the image of $\rho$, so it suffices to produce a map
 \w{\xi_{\fX}:\VE\fX=\fME\RE\rho \fX \to \fX} fitting into a commutative diagram:
\mydiagram[\label{eqalg}]{
  \VE \VE\fX \ar[d]_{\xi_{\VE\fX}} \ar[rr]^{\VE(\xi_\fX)} && \VE\fX\ar[d]^{\xi_\fX} \\
  \VE\fX \ar[rr]_{\xi_\fX} & & \fX}
\noindent in \w[,]{\ST} and then set
\w[\vsn.]{\zeta_{\rho \fX} =(\rho\xi\sb{\fX})\op}

 \noindent\textbf{Step 1.}\hsm
If we let \w[,]{\bK:=\RE \rho \fX} by \wref{rightadj} we have
\begin{myeq}\label{eqre}
\bK~=~\prod\limits_{\bM \in \TE} \prod\limits_{\phi \in \Sp(\bS{0}, \fX\{ \bM\})}
 \prod\limits\sb{j: \bM \to \bM'} \bQ\sp{(\bM, \phi, j)}
\end{myeq}
\noindent which is in \w[.]{\Thel{\kappa}} Thus we have an indexing category
$$
 \II~=~\coprod\limits_{\bM \in \TE} \coprod\limits_{\phi \in \Sp(\bS{0}, \fX\{ \bM\})}
 \coprod\limits\sb{j: \bM \to \bM'}\ \II\sp{(\bM, \phi, j)}
$$
 \noindent (depending on $\fX$), and functors
 \w{\hP\sp{(\bM, \phi, j)}:\II\sp{(\bM, \phi, j)}\to\TE} such that
 \w{\lim \hP\sp{(\bM, \phi, j)}=\bQ\sp{(\bM, \phi, j)}} as in \wref[.]{dualstover}

We can describe the indexing category \w{\II\sp{(\bM, \phi, j)}} by:
\myrdiag[\label{ddualstover}]{
  \coprod\limits_{\Phi\in \chi^{-1}(\phi)} (\Phi) \ar[d]^{\coprod \pi\sb{\Phi}}
  \ar[rr]^{\coprod (\gamma\sb{j})} &&
  \coprod\limits_{\Phi\in \chi^{-1}(\phi)} (\Phi)' \ar[d]^{\coprod \pi'\sb{\Phi}}  &&
  \coprod\limits_{\Psi\in(\chi')^{-1}(b\sb{j}(\phi))} (\Psi) \ar[d]^{\coprod \pi'\sb{\Psi}}
  \ar[ll]_{\bot (e\sb{j}\sp{\ast})} \\
  \coprod\limits_{\Phi\in \chi^{-1}(\phi)} (\Phi\sb{b}) \ar[rr]^{\coprod (\delta\sb{j})} &&
  \coprod\limits_{\Phi\in \chi^{-1}(\phi)} (\Phi\sb{b})' &&
  \coprod\limits_{\Psi\in(\chi')^{-1}(b\sb{j}(\phi))} (\Psi\sb{b})
  \ar[ll]_{\bot (e\sb{j})} \\
  (b) \ar[u]_-{\diag}\ar[rr]^{b\sb{j}}&& (b') \ar[u]^-{\diag} \ar[urr]_-{\diag} &&
  }
where \w{\coprod_{s\in S}\,(s)} is a discrete subcategory with object set $S$, and
\w{\diag:(b)\to\coprod_{s\in S}\,(s)} means that there is a single arrow from \w{(b)}
to each \w[.]{(s)}

The notation \w[,]{(\Phi)'} and so on, is intended to distinguish objects in
different discrete categories with the same set of indices \w[.]{\chi^{-1}(\phi)}
The notation \w{\coprod \pi\sb{\Phi}} for a map between such categories means that
each object \w{(\Phi)} in the upper left corner maps to the
corresponding \w{(\Phi\sb{b})} beneath it.
The reader should keep in mind the motivating functor from \wref{ddualstover} to
\w[,]{\Sp} described in \wref[.]{sa}

The somewhat nonstandard notation
$$
\bot(e\sb{j}\sp{\ast}):\coprod\limits_{\Psi\in(\chi')^{-1}(b\sb{j}(\phi))}\ (\Psi)~\to~
\coprod\limits_{\Phi\in \chi^{-1}(\phi)}\ (\Phi)'
$$
\noindent means that if \w{\Psi=e\sb{j}(\Phi)} then \w{(\Psi)} is sent to
\w{(\Phi)} in the second discrete subcategory.

The functor \w{\hP=\hP\sp{(\bM, \phi, j)}:\II\sp{(\bM, \phi, j)}\to\TE} is described
implicitly by \wref[:]{dualstover} thus \w{\hP((\Phi))=P\bM} for each
\w[,]{\Phi\in \chi^{-1}(\phi)} and so on.
The top right left-facing arrow in \wref{dualstover} maps into the copy of
\w{P\bM'} indexed by $\Phi$ (in the top central product) by projecting the product
in the top right onto the factor \w{P\bM'} indexed by \w[.]{e\sb{j}(\Psi)}

The functors \w{\hP\sp{(\bM, \phi, j)}} fit together to define
\w[,]{\hP:\II\to\TE} with
\w[\vsn.]{\bK~=~\lim\sb{f\in\II}\,\hP(f)}

\noindent\textbf{Step 2.}\hsm
To define the map \w[,]{\xi\sb{\fX}:\VE\fX\to\fX} note that since
\w[,]{\VE\fX=\fME\bK}  by Lemma \ref{yoneda} \w{\xi\sb{\fX}} should correspond to
 the value of \w{\xi\sb{\fX}(\Id\sb{\bK})} in \w[.]{\Sp(\bS{0},\fX\lin{\bK})} But
 \w[,]{\fX\lin{\bK}=\fX\lin{\lim\sb{f\in\II}\,\hP(f)}=\lim\sb{f\in\II}\fX\lin{\hP(f)}}
 because the mapping algebra $\fX$ commutes by definition with the limits  in
 \wref[.]{eqre} Thus we may define \w{\xi\sb{\fX}(\Id\sb{\bK})} to be the
 tautological map whose values at \w{\fX\lin{\hP(f)}} is $f$ itself\vsn.

 \noindent\textbf{Step 3.}\hsm
 A similar calculation shows that \w{\bL:=\RE \rho \VE\fX}
 is a limit of a functor \w[,]{\hN=\hN\sb{\VE\fX}:\JJ\to\TE}
 but in this case the
 indexing category $\JJ$ can be described somewhat more explicitly because
 \w{\VE\fX=\rho\fME\bK} is also representable. Thus
\begin{myeq}\label{eqrell}
\bL~=~\prod\limits_{\bM \in \TE} \prod\limits_{\phi:\bK\to\bM}
 \prod\limits\sb{j: \bM \to \bM'} \hN \sp{(\bM, \phi, j)}
\end{myeq}
\noindent which again is in \w[.]{\Thel{\kappa}} Therefore,
\w[,]{\JJ~=~\coprod\limits_{\bM \in \TE} \coprod\limits_{\phi:\bK\to\bM}
 \coprod\limits\sb{j: \bM \to \bM'}\ \JJ\sp{(\bM, \phi, j)}}
where \w{\JJ\sp{(\bM, \phi, j)}} defined analogously to \wref[,]{ddualstover}
and thus the factor \w{\hN \sp{(\bM, \phi, j)}} in \wref{eqrell} (for nullhomotopic
\w[)]{\phi:\bK\to\bM} is the limit of the diagram:
\myrdiag[\label{dualstoverj}]{
  \prod\limits_{\Phi:\phi\sim\ast}\ P\bM \ar[d]^{\prod p\sb{\bM}} \ar[rr]^{\prod Pj} &&
  \prod\limits_{\Phi:\phi\sim\ast} P\bM' \ar[d]^{\prod p\sb{\bM'}}  &&
  \prod\limits_{\Psi:j\circ \phi\sim\ast} P\bM'
  \ar[ll]_{\top \proj_{j\circ\Phi}}\ar[d]^{\prod p\sb{\bM'}} \\
  \prod\limits_{\Phi:\phi\sim\ast} \bM \ar[rr]^{\prod j} &&
  \prod\limits_{\Phi:\phi\sim\ast} \bM'  &&
  \prod\limits_{\Psi:j\circ \phi\sim\ast} \bM' \ar[ll]_{\top \proj_{j\circ\Phi}}\\
  \bM \ar[u]_-{\diag}\ar[rr]^{j} && \bM' \ar[u]^-{\diag} \ar[urr]_-{\diag} &&
  }
\noindent where we have already taken the limits over the discrete subcategories of
\w[.]{\JJ\sp{(\bM, \phi, j)}}

Note that the objects of \w{\JJ\sp{(\bM, \phi, j)}} are actual spectrum maps $g$ from
$\bK$ into the value of $\hN$ at this object, namely \w[,]{\hN(g)} which is
always one of \w[.]{\{\bM,\bM',P\bM,P\bM'\}}
The map of mapping algebras
\w{\xi\sb{\VE\fX}~=~\fME\eta\sb{\bK}~:~\fME\bL~\to~\fME\bK}
(see Remark \ref{rcoalg}) corresponds under Lemma \ref{yoneda} to the tautological map
 \w{\eta\sb{\bK}:\bK\to\lim\sb{\JJ}\hN} which sends $\bK$ into \w{\hN(g)} by $g$
 itself\vsn.

\noindent\textbf{Step 4.}\hsm
Similarly, the composite \w{\xi\sb\fX\circ \xi\sb{\VE \fX}~:~\VE \VE\fX~\to~\fX}
corresponds under Lemma \ref{yoneda} to the value of
 \w{\xi\sb\fX\circ \xi\sb{\VE \fX}(\Id\sb{\bL})} as a spectrum map
 \w[,]{\bS{0}\to\fX\lin{\bL}} where again
\w[.]{\fX\lin{\bL}~=~\fX\lin{\lim\sb{g\in\JJ}\,\hN(g)}~=~\lim\sb{g\in\JJ}\fX\lin{\hN(g)}}
Since we are mapping into a limit, this is uniquely determined by
the map \w{\bS{0}\to\fX\lin{\hN(g)}} for various \w[,]{g\in\Obj(\JJ)}
given by \w[\vsn.]{g\sb{\ast}\xi\sb{\fX}(\Id\sb{\bK})}

\noindent\textbf{Step 5.}\hsm
By definition, the map \w{\VE(\xi\sb{\fX})=\fME\RE \rho \xi\sb{\fX}~:~\VE \VE\fX~\to~\VE\fX}
is induced by
\begin{myeq}\label{eqvarphi}
\RE \rho \xi\sb{\fX}~=~\bK~\xra{\varphi}~\cTE\bK~=~\bL~=~\lim\sb{g\in\JJ}\,\hN(g)~.
\end{myeq}
Again, we are mapping into a limit, so this is uniquely determined by
maps \w{\varphi\sb{g}:\bK\to\hN(g)} for various \w[.]{g\in\Obj(\JJ)}
However \w[,]{\bK=\lim\sb{f\in\II}\,\hP(f)} so it has structure maps to
its constituents \and we see that \w{\varphi\sb{g}} is precisely the
structure map \w{\pi\sb{f}:\bK\to\hP(f)=\hN(g)} where
\w[\vsn.]{f=g\sb{\ast}\xi\sb{\fX}(\Id\sb{\bK})\in\fX\lin{\hN(g)}}

\noindent\textbf{Step 6.}\hsm
Finally, the map \w{\xi\sb{\fX}\circ\VE(\xi\sb{\fX})~:~\VE \VE\fX~\to~\VE\fX}
is the composite of the two
maps given in Steps 2 and 5, respectively. It corresponds under Lemma \ref{yoneda}
to the map \w{\psi:\bS{0}\to\fX\lin{\bL}} which is the image of
\w{\xi\sb{\fX}(\Id\sb{\bK}):\bS{0}\to\fX\lin{\bK}} under the map $\varphi$
of \wref[.]{eqvarphi}

However,
\w[,]{\Id\sb{\bK}:\bK\to\bK=\lim\sb{f\in\II}\,\hP(f)} as a
map into a limit, is determined by the structure maps
\w[,]{\pi\sb{f}:\bK\to\hP(f)} where
\w{\xi\sb{\fX}(\pi\sb{f}):\bS{0}\to\fX\lin{\hP(f)}} is
given by $f$ itself.

Since \w{\fX\lin{\bL}=\lim\sb{g\in\JJ}\fX\lin{\hN(g)}} is a limit,
it is enough to describe the component of $\psi$ into each constituent
\w[.]{\fX\lin{\hN(g)}} where it is given by the structure map
\w{\pi\sb{f}:\bK\to\hP(f)} for
\w[.]{f=g\sb{\ast}\xi\sb{\fX}(\Id\sb{\bK})\in\fX\lin{\hN(g)}}
Thus $\psi$ is determined in this component by
\w{g\sb{\ast}\xi\sb{\fX}(\Id\sb{\bK})} \wwh the same value we got in  Step 4.

This shows that \w{\xi\sb{\fX}\circ\VE(\xi\sb{\fX})} indeed equals
 \w[.]{\xi\sb\fX\circ \xi\sb{\VE \fX}}
 \end{proof}

For the representable mapping algebra \w[,]{\fX=\fME\bY} Proposition \ref{coalgebra}
and Remark \ref{rcoalg} yield:

\begin{cor}\label{free_coalgebra}
  The coalgebra map \w{\zeta} for the arrow set \w{\LE \bY} is induced by a map of
  mapping algebras \w[,]{\zeta':\fME(\RE(\rho\fME\bY)) \to\fME\bY} so that
  \w[.]{\zeta = (\rho \circ \zeta')\op}
\end{cor}

%
%
\sect{Small mapping algebras}
\label{csma}

As noted in Section \ref{cdsc}, our goal is to associate to any mapping algebra
$\fX$ a cosimplicial resolution \w[,]{\Wu} with \w{\bY=\Tot\Wu} \emph{realizing} $\fX$:
that is, having $\fX$ weakly equivalent (\S \ref{smcs}) to \w[.]{\fmE\bY}

In order to show this, using \cite{Bou03}, \w{\bY\to\Wu}
must be acyclic with respect to \emph{any} $\bE$-module $\bM$. However, even if
\w{\fX=\fME\bY} to begin with, the modules appearing in \w{\TE} are of bounded
cardinality, so for general $\bE$, merely iterating the monad \w{\cTE} on $\bY$ to
produce a coaugmented cosimplicial space \w{\bY\to\Wu} will not yield the required
resolution  (although for \w[,]{\bE=\HFp} this can be done, as in
\cite[\S 3]{BS18}).

To bypass this difficulty, in this section we will show that
given $\fX$, there is a cardinal $\lambda$ such that any map from each spectrum
\w{\bW\sp{n}} to $\bM$ factors through a module in \w[.]{\TE}

\begin{defn}\label{indmodulemap}
  Given \w{\bX\in\Sp} and \w[,]{\bM\in\ModE} any map \w{\phi : \bX \to \bM} in \w{\Sp}
  is adjoint to an $\bE$-module map\w{\widetilde{\phi}:\bE\otimes\bX\to\bM}
  with \w[,]{\mu \sb{\bM}\circ (Id\sb{\bE} \otimes \phi)=\widetilde{\phi}} where \w{\mu \sb{\bM}}
  is the module structure map. In symmetric spectra the map $\widetilde{\phi}$ has an image
  \w{\im(\widetilde{\phi})} inside $\bM$, of cardinality \www[.]{\leq \|\bM\|}
  Since $\widetilde{\phi}$ is an $\bE$-module map, it fits into
  a commutative diagram

\begin{myeq}\label{immodule}
  \xymatrix{\bE\otimes (\bE \otimes \bX) \ar[d]_{\mu_{\bE\otimes \bX}}
    \ar[r]^{\Id \sb{\bE} \otimes \widetilde{\phi}} &
  \bE \otimes \im(\widetilde{\phi})\ar[d]^{\mu_{\bM}} \\
  \bE\otimes \bX \ar[r]^{\widetilde{\phi}} & \bM}
\end{myeq}

It follows that the image of \w{\mu \sb{\bM}: \bE \otimes \im(\widetilde{\phi}) \to \bM}
sits inside \w[,]{\im(\widetilde{\phi})} so the latter has an $\bE$-module structure.

We say that $\phi$ is \emph{effectively surjective} if \w[,]{\bM = \im(\widetilde{\phi})}
and denote the set of such maps by \w[.]{\eSp(\bX, \bM)}

If \w{\Phi: \bX \to P\bM} is a nullhomotopy of \w[,]{\phi: \bX \to \bM} with
\w{p \sb{\bM} \circ \Phi = \phi} (where \w{p \sb{\bM}} is the path fibration,
an $\bE$-module map), then  \wref{eqfunmap} yields a
map \w[.]{\Phi' : \bX \otimes \Delta[1] \sb{+} \to \bM}
Define an \ww{\bE}-module map
\w{\widehat{\Phi}: \bE \otimes \bX\otimes \Delta[1]\sb{+}\to \bM} by setting
\w[.]{\widehat{\Phi}:= \mu\sb{\bM} \circ (Id\sb{\bE} \otimes \Phi')}
We say that $\Phi$ is an \emph{effectively surjective nullhomotopy} of $\phi$ if
\w[.]{\bM = \im (\widehat{\Phi})} The \ww{\bE}-module structure on
\w{\im(\widehat{\Phi})} is given by \wref[.]{immodule}

Note that if  $\phi$ is effectively surjective,
so is $\Phi$. We denote the set of effectively surjective nullhomotopies of
\w{\phi: \bX \to \bM} by \w[.]{\eSp(\bX, P\bM)_\phi}
If we define \w{\overline{\Phi}: \bE \otimes \bX \to P \bM} by
\w[,]{\overline{\Phi}(e \otimes x):= \mu\sb{\bM}(e \otimes \Phi(x)(-))}
we see that \w{\overline{\Phi}} is a nullhomotopy of \w[.]{\widetilde{\phi}}
\end{defn}

Our goal is now to modify the construction \wref{dualstover} used in defining \w{\RE}
in terms of effectively surjective maps and nullhomotopies alone, thus
obtaining a modified
version of \w[:]{\cTE}

\begin{defn}
  For any $\bE$-module $\bM $ and effectively surjective \w[,]{\phi: \bX \to  \bM}
  define \w{\bQ_{\phi}} to be the pullback in $\Sp$:
$$
\xymatrix{
  \bQ_{\phi} \ar[rr]\ar[d] &&
{\displaystyle \prod\sb{\textstyle{\mbox{\small $(j:\bM\to\bM')\in\ModE$}}}}
\hspace{-15mm}'\hspace{20mm}
{\displaystyle \prod\sb{\textstyle{\mbox{\small $\eSp(\bX,P\bM')\sb{j\phi}$}}}}
\ P\bM' \ar@<18ex>[d]\sp<<<<{\prod' p\sb{\bM'}}\\
\bM \ar[rr]^-{(j)} & &
{\displaystyle \prod\sb{\textstyle{\mbox{\small $(j:\bM\to\bM')\in\ModE$}}}}
\hspace{-15mm}'\hspace{20mm}
{\displaystyle \prod\sb{\textstyle{\mbox{\small $\eSp(\bX,P\bM')\sb{j\phi}$}}}}
\bM' }
$$
\noindent where \w{\prod'} indicates that empty factors are to be omitted from the
product, so that the limit is in fact taken over a small diagram.

Finally, set
\w[.]{\sTE \bX: = \prod'_{\bM\in\ModE} \; \prod\limits_{\phi \in \eSp(\bX, \bM)} \bQ_{\phi}}
\end{defn}

\begin{defn}\label{dlambdax}
  For any symmetric spectrum $\bX$ we define a cardinal
$$
  \lambda_\bX~:=~\sup_{\bM \in \ModE} \{\|\im(\widetilde{\phi})\|: \phi: \bX \to \bM \} \cup
  \{\|\im(\widehat{\Phi})\|: \Phi: \bX \to P\bM\}.
$$
This makes sense since  \w{\|\im(\widetilde{\phi})\|} and
\w{\|\im(\widehat{\Phi})\|} are bounded by \w{\|\bE\otimes\bX\|} and
\w[,]{\|\bE\otimes\bX\otimes\Delta[1]\sb{+}\|} respectively. Thus for all
practical purposes we may simply set
\w[.]{\lambda_\bX:=\|\bE\otimes\bX\otimes\Delta[1]\sb{+}\|}
\end{defn}

\begin{prop}\label{card}
  For any symmetric spectrum $\bX$ and \w{\kappa \geq \lambda_\bX} we have a canonical
  isomorphism \w[.]{\sTE\bX \cong \Rel{\kappa}\Lel{\kappa} \bX}
\end{prop}

\begin{proof}
  Recall from \S \ref{rmonad} that we write \w{\Tel{\kappa}} for
  \w[.]{\Rel{\kappa}\Lel{\kappa}}
  By the description in Lemma \ref{lrightadj}, we know that
  \w{\Tel{\kappa} \bX= \Rel{\kappa}(\rho \circ \F(\bX, -))} is given by
  \begin{myeq}\label{eqrightadj}
    \Tel{\kappa}\bX~:=~
  \prod\limits_{\bM \in \TE} \prod\limits_{\phi:\bX\to\bM}
  \prod\limits_{j: \bM \to \bM'} \bQ^{(\bM,\phi, j)}~.
\end{myeq}
  \noindent where for nullhomotopic \w{\phi:\bX\to\bM} the $\bE$-module
  \w{\bQ^{(\bM,\phi, j)}} is the limit of:
$$
  \xymatrix{
    \prod\limits_{\Phi:\phi\sim\ast}P\bM \ar[r]^{\prod Pj}\ar[dr]_{\prod j p \sb{\bM}} &
    \prod\limits_{\Phi:\phi\sim\ast}P\bM^{\prime}\ar[d]^{\prod p\sb{\bM'}} &&
    \prod\limits_{\Psi:j \circ \phi\sim\ast}P\bM'\ar[ll]_{\top \proj\sb{j\circ\Phi}}
    \ar[d]^{\prod p\sb{\bM'}} \\ & \prod\limits_{\Phi:\phi\sim\ast} \bM' &&
    \prod\limits_{\Psi:j \circ \phi\sim\ast}\bM' \ar[ll]_{\top \proj\sb{j\circ\Phi}} &
  \bM \ar[l]_-{\diag\circ j} }
$$
(compare \wref[).]{dualstoverj}

Our goal is to replace this limit by one involving only $\bE$-modules in
 \w[,]{\Thel{\kappa}} by using only effective surjective maps and nullhomotopies.

  Note that
  \w[,]{\{\Psi:j \circ \phi\sim\ast\} = j_*\{\Phi:\phi\sim\ast\}\amalg\New^0_1}
    where \w{\New^0_1} is the set of nullhomotopies of \w{j \circ\phi} not induced via $j$
    from nullhomotopies of $\phi$.

If \w{\phi: \bX \to \bM} is effectively surjective, then so is any nullhomotopy
 \w{\Phi: \bX \to P\bM} of $\phi$.  So we may replace the index set
 \w{\{\Phi:\phi\sim\ast\}} by \w[.]{\eSp(\bX, P\bM)_{\phi}}

 If \w{\Phi : \bX \to P\bM'} is a nullhomotopy of \w{j\circ\phi : \bX \to \bM'}
 which is not effectively surjective, we have a commutative diagram
$$
   \xymatrix{\bX \ar@/^2pc/[rr]^{\Phi} \ar[r]^{\eta \sb{\bX}}\ar[dr]_{\Phi''}&
     \bE \otimes \bX \ar[r]^{\overline{\Phi}}\ar[d] &P\bM' \\& P\bM'' \ar[ur]_{Pj'}~,}
   $$
   where\w{\bM''=\im(\widehat{\Phi}: \bE\otimes \bX \otimes \Delta[1]\sb{+} \to \bM')} (see \S\ref{indmodulemap}) and \w{j' : \bM'' \to \bM'} is the inclusion. Thus
\w{\Phi''} is an effectively surjective nullhomotopy.

   Thus, whenever \w[,]{\kappa > \lambda \sb{\bX}} we have a cofinal diagram defining
   \w{\Tel{\kappa} \bX} in which only those \w{\bM\in\Thel{\kappa}} appear
   for which there is either an effective surjection or an effectively surjective
   nullhomotopy for some map \w[.]{\bX \to \bM} Therefore, we may restrict ourselves to
  $\bM$ in \w[.]{\Thel{\lambda \sb{\bX}}} This shows that the natural map
   \w{\sTE \bX \to\Tel{\kappa} \bX} is an isomorphism.
\end{proof}

\begin{remark}\label{rcard}
  Proposition \ref{card} shows that \w{\sTE} is in fact locally small, in that for
  every \w[,]{\bX\in\Sp} \w{\sTE\bX} is naturally equivalent to the value of a small functor.

  In particular, this implies that \w[,]{\sTE} \emph{a posteriori}, is a functor, since for
  any map \w[,]{f:\bX\to\bY} we have \w{f\sb{\ast}:=\sTE(f)=\Tel{\kappa}(f)} for
\w{\kappa=\max\{\lambda\sb{\bX},\lambda\sb{\bY}\}} (and similarly for composites).
However, the reader may find the following explicit description of \w{f\sb{\ast}}
helpful:

Let  \w{\psi: \bY \to \bM} an effective surjection and \w{j: \bM \to \bM'} a map of
  $\bE$-module spectra, with \w{\Psi: \bY \to P\bM'} an effectively surjective nullhomotopy
  of \w[.]{j \circ \psi} Set  \w{\bM''= \im(\widetilde{\psi \circ f})} and
  \w[.]{\bM'''= \im(\widehat{\Psi \circ f})} By \wref{immodule} it follows that
  \w{\bM''} and \w{\bM'''} are \ww{\bE}-modules.

Note that we have the following commutative diagram
$$
\xymatrix{
  \bX\ar[d]\sb{\eta\sb{\bX}} \ar[r]^{f} & \bY\ar[d]\sp(0.6){\eta\sb{\bY}} \ar[r]^{\psi} & \bM \\
  \bE \otimes \bX \ar[urr]^(0.35){\widetilde{\psi\circ f}} \ar[r]_{\Id \sb{\bE} \otimes f} & \bE \otimes \bY \ar[ur]_(0.6){\widetilde{\psi}} \ar[r]_{Id \sb{\bE} \otimes \psi} &
  \bE \otimes \bM \ar[u]_{\mu \sb{\bM}}}
$$

\noindent in \w[.]{\Sp} Set \w[.]{\phi= \psi \circ f} By the definitions of \w{\bM''}
and \w{\bM'''} we get \ww{\bE}-module  maps \w{j'': \bM'' \to \bM} and
\w{j': \bM''' \to \bM'} fitting into the diagram
$$
\xymatrix{
  \bE \otimes \bY \ar@/^3pc/[rrrr]^{\overline{\Psi}}\ar[ddrr]_{\widetilde{\psi}} &
  \bE \otimes \bX \ar[l]_{Id \sb{\bE} \otimes f}\ar[rr]^{\overline{\Phi}}\ar[dr]_{\widetilde{\phi}} & &
  P\bM'''\ar[d]^{p \sb{\bM'''}} \ar[r]^{Pj'} & P\bM' \ar[dd]^{p \sb{\bM'}}\\
  && \bM'' \ar[r]^{j'''}\ar[d]^{j''} &\bM''' \ar[dr]^{j'}& \\ && \bM \ar[rr]^{j} & & \bM'}
$$
The map \w{j'''} exists and it is  an \ww{\bE}-module map because
\w[.]{\widehat{\Psi \circ f} \circ i\sp{\bX}\sb{1} = \widetilde{j\circ\phi}}
Here \w{i\sp{\bX}\sb{1}} is given by the identification of $\bX$ with
\w{\bX \otimes \lin{1}} inside \w[.]{\bX\otimes \Delta[1]\sb{+}}

The component $\vartheta$ of the map \w{f\sb{\ast}:\sTE \bX \to \sTE \bY} into the factor
\w{\bQ_{\psi}} of \w{\sTE \bY} is defined by projecting from
\w{\sTE \bX} onto \w{\bQ_\phi} and onto the copy of \w{P\bM'''} indexed by $\Phi$
($=\overline{\Phi}\circ \eta\sb{\bX}$).
This then maps by \w{Pj'} to the copy of \w{P\bM'} in \w{\bQ_{\psi}} indexed by
$\Psi$.

The map \w[,]{f\sb{\ast}} restricted to \w[,]{\bQ\sb{\phi}} is then given
by the universal property of the pull-back square as follows:

$$
\xymatrix{\bQ\sb{\phi}\ar[drrr]^{\vartheta} \ar[d] \ar[dr]\sp(0.6){f\sb{\ast}}\\
\bM''\ar[dr]\sb{j''} & \bQ_{\psi} \ar[rr]\ar[d] &&
{\displaystyle \prod\sb{\textstyle{\mbox{\small $(j:\bM\to\bM')\in\ModE$}}}}
\hspace{-15mm}'\hspace{20mm}
{\displaystyle \prod\sb{\textstyle{\mbox{\small $\eSp(\bX,P\bM')\sb{j\psi}$}}}}
\ P\bM' \ar@<18ex>[d]\sp<<<<{\prod' p\sb{\bM'}}\\
& \bM \ar[rr]^-{(j)} & &
{\displaystyle \prod\sb{\textstyle{\mbox{\small $(j:\bM\to\bM')\in\ModE$}}}}
\hspace{-15mm}'\hspace{20mm}
{\displaystyle \prod\sb{\textstyle{\mbox{\small $\eSp(\bX,P\bM')\sb{j\psi}$}}}}
\bM'~. }
$$
\end{remark}

%
%
\sect{Cosimplicial resolutions and the $\bE$-based Adams spectral sequence}
\label{ccsass}

For any limit cardinal $\lambda$, the adjoint functors \w{\LE} and
\w{\RE} constructed in Section \ref{cdsc} define a
comonad \w{\SE=\LE\circ\RE} on the category \w{(\Xi\sp{\lambda})\op}
(see \S \ref{rxi}). Using \cite[5.7,8.5,9.7]{Bou03}, we now show how this comonad,
applied to a mapping algebra $\fX$, yields a cosimplicial spectrum \w{\Wu} such that
\w{\Tot\Wu} realizes $\fX$ under favorable circumstances
(in particular, when \w{\fX=\fME\bY} for an $\bE$-good spectrum $\bY$).

We note that the proper setting for our constructions is the resolution model category of
cosimplicial spectra of \cite[\S 3]{Bou03}, and the associated model category of
simplicial mapping algebras (see \S \ref{srmc} below).

\begin{mysubsection}{The cosimplicial spectrum $\Wu$ associated to $\fX$}
\label{scsfx}
Given a mapping algebra \w[,]{\fX\in\MapE} by iterating the comonad
\w{\SE} on the arrow set \w{A=\rho \fX} we obtain
as usual an augmented simplicial object \w{\vare:\Vd\to A} in \w[,]{\Xi\op}
with \w[,]{\wV\sb{k}:= (\SE)^{k+1} A} and face and degeneracy maps induced by the
structure maps of the comonad (see \cite[8.6.4]{Wei94}).

If we assume that $\fX$ extends as in Proposition \ref{coalgebra} \wh e.g., if it
is representable \wh then \w{A=\rho \fX} has a coalgebra structure
\w{\zeta_{A} : A \to \SE A = \wV_0} over the comonad \w[,]{\SE}
which provides an extra degeneracy for \w[.]{\Vd\to A} Thus \w{\RE} applied to
this augmented simplicial object yields a cosimplicial spectrum \w[,]{\Wu}
with \w[,]{\bW\sp{0}=\RE(A)} \w[,]{\bW\sp{1}=\RE(\wV\sb{0})}
\w[,]{d\sp{0}=\RE(\zeta\sb{A})} and \w{d\sp{1}=\RE(\vare)}
(see \cite[Prop. 3.27]{BS18} for a detailed description). By applying the functor
\w{\LE} to this cosimplicial spectrum we obtain a simplicial object in mapping
algebras \w{\fME \Wu} (by contravariance of \w[),]{\RE} which is augmented to $\fX$,
yielding a map of simplicial mapping algebras
\w[.]{\fME \Wu \to c(\fX)_{\bullet}}
\end{mysubsection}

\begin{defn}\label{detwe}
We say that a map \w{\ff:\fWd\to\fUd} of simplicial spectral functors (e.g.,
mapping algebras) is an \emph{\ww{E\sp{2}}-equivalence} (cf.\ \cite{Jar04})
if for every \w[,]{\bM \in \TE} the induced map of simplicial
abelian groups \w{\fWd\lin{\bM}\to\fUd\lin{\bM}} is a weak equivalence (of simplicial
sets).
\end{defn}

\begin{prop}\label{acyclic_reso}
If for \w{\fX\in\MapE} and \w{\Wu} as above $\fX$ is known to be a
homotopy functor, then \w{\fME \Wu \to c(\fX)_{\bullet}} is an \ww{E\sp{2}}-equivalence.
\end{prop}

\noindent Equivalently, for every \w[,]{\bM \in \TE} the augmented
simplicial abelian group \w{[\Wu, \bM]\to\pi_0(\fX\lin{\bM})} is
acyclic, where \w{[\Wu, \bM]} is the simplicial abelian group
obtained by applying the homotopy functor \w{[-,\bM]} (see Lemma
\ref{lerel}) in each cosimplicial dimension.
\begin{proof}
By standard facts about comonads (see \cite[Proposition 8.6.10]{Wei94}), the
augmented simplicial arrow set \w{\LE \Wu \to \rho \fX} is contractible,
so by Corollary \ref{grpstr} and \wref{eqomegapi} the augmented simplicial mapping algebra
\w{\fME\Wu\to\fX} is contractible, too.
\end{proof}

\begin{remark}\label{rdegmod}
  Note that each \w{\bW^n} is an $\bE$-module, and for each \w[,]{0\leq i\leq n}
the codegeneracy map \w{s^i_n:\bW^n\to\bW^{n+1}} is \w[,]{\RE (\SE)^{n-i} \epsilon_{(\SE)^i A}}
where \w{\epsilon_{(\SE)^i A}} is the comonad counit map for \w[.]{(\SE)^i A}
Thus the codegeneracies are in the image of \w{\RE} and in particular are
$\bE$-module maps.
\end{remark}

\begin{defn}\label{dgeinj}
For any ring spectrum $\bE$, \w{\G(\bE):=\ModE} is a class of injective models in
\w{\Sp} in the sense of \cite[\S 3.1]{Bou03}, and we have a \ww{\G(\bE)}-localization
functor \w[,]{\widehat{\cL}_{\G(\bE)}:\Sp\to\Sp} with a map
\w{\eta\sb{\bY}:\bY\to\widehat{\cL}_{\G(\bE)}\bY} (see \cite[\S 8]{Bou03}).

A symmetric spectrum $\bY$ is called \ww{\bE}-\emph{good} if \w{\eta\sb{\bY}}
is an $\bE$-\emph{equivalence} \wh that is, for each \w[,]{\bM \in \G(\bE)}
the induced map \w{[\widehat{\cL}_{\G(\bE)}\bY, \bM] \to [\bY, \bM]} is an isomorphism
(this is called a \ww{\G(\bE)}-equivalence in \cite{Bou03}).
\end{defn}

\begin{remark}\label{recompl}
  By \cite[Theorems 6.5 \& 6.6]{Bou79}, when $\bE$ and $\bY$ are connective and
  the core $R$ of \w{\pi\sb{0}\bE} is either \w{\ZZ/n} or a subring of $\QQ$,
  \w{\widehat{\cL}_{\G(\bE)}\bY} is simply the usual $R$-completion of $\bY$, given by
    smashing with the Moore spectrum for $R$ (see \cite[\S 2]{Bou79}).
\end{remark}

\begin{notn}\label{cardY}
For any \w[,]{\bY\in\Sp} let
\w[,]{\wlY:=\sup\{\lambda_{\sTE^n \bY}\}\sb{n\in\NN}} in the notation of \S \ref{dlambdax}.
\end{notn}

\begin{mysubsection}{The cosimplicial spectrum $\Wu$ associated to $\bY$}
\label{scsby}
When the mapping algebra $\fX$ of \S \ref{scsfx} is realizable by a
spectrum $\bY$, and \w[,]{\lambda\geq\wlY} we can think of the
cosimplicial spectrum \w{\Wu} constructed there from \w{\fX=\fME\bY}
as having the form \w[,]{\bW\sp{k}:=\sTE\sp{k+1}\bY} with
coaugmentation \w[.]{\eta\sb{\bY}:\bY\to\sTE\bY}

For a cosimplicial spectrum \w{\Wu} the \emph{totalization} \w{\Tot\Wu} as
in \cite[Section 2.8]{Bou03} then satisfies
\end{mysubsection}

\begin{thm}\label{main1}
If $\bE$ is a ring spectrum, $\bY$ an \ww{\bE}-good symmetric spectrum,
\w[,]{\lambda=\wlY} and \w{\Wu} is as above, the canonical map \w{\bY \to \Tot \Wu} is an
$\bE$-equivalence.
\end{thm}

\begin{proof}
By Proposition \ref{card}, the augmented simplicial group \w{[\Wu, \bM] \to [\bY, \bM]}
is acyclic for all \w[,]{\bM \in \G(\bE)} using Proposition \ref{acyclic_reso}.
Since $\bY$ is \ww{\bE}-good, \w{\widehat{\cL}_{\G(\bE)}\bY\simeq\Tot \Wu} so
\w{\bY \to \Tot \Wu} is an $\bE$-equivalence by \cite[\S 9]{Bou03}.
\end{proof}

\begin{mysubsection}{Cosimplicial Adams resolutions}
\label{scsar}
Recall that an $\bE$-\emph{Adams resolution} for an ($\bE$-good) spectrum $\bY$
is a sequence of spectra
\w{\xymatrix{
\bX=\bX_0 & \bX_1 \ar[l]^-{g_0} & \bX_2 \ar[l]^-{g_1} & \cdots \ar[l] }}
such that for each \w[:]{s \geq 0}

\begin{enumerate}
\renewcommand{\labelenumi}{(\roman{enumi})}
\item \w{\holim\bX\sb{s}} is $\bE$-equivalent to $\bY$.
\item If \w{\bK_s} is the cofiber of \w{g\sb{s}} and
  \w{f_s: \bX_s \to \bK_s} is the structure map, then \w{\bE \otimes f_s}
has a retraction.
\item \w{\bK_s} is a retract of \w[.]{\bE \otimes \bK_s}
\end{enumerate}
(see \cite[\S 2.2.1]{Rav86}).

Given an $\bE$-good spectrum $\bY$ with $\bE$-mapping algebra \w[,]{\fME\bY}
we saw in the previous section  how to construct a cosimplicial spectrum \w{\Wu}
such that \w{\Tot \Wu} is an $\bE$-completion of $\bY$, in the sense
of \cite[\S 2.2.2]{Rav86}.

Note that we have a model category of $\bE$-modules given by
\cite[Theorem 4.1]{SShipA}, and thus an induced Reedy model
category \w{\ModE\sp{\Delta}} of cosimplicial $\bE$-modules (see
\cite[Theorem 15.3.4]{Hir03}). We may thus replace the \w{\Wu} of
\S \ref{scsfx} by a Reedy fibrant object in
\w[,]{\ModE\sp{\Delta}} (which we also denote by \w[,]{\Wu} to
avoid unnecessary notation).

We then have a tower of fibrations
\begin{myeq}\label{eqtottower}
\xymatrix{\bW^0 = \Tot_0(\Wu) & \Tot_1(\Wu) \ar[l]_(0.55){h_1} \cdots \Tot_{k-1}(\Wu) &
  \Tot_k(\Wu) \ar[l]_(0.3){h_{k}} & \cdots\ar[l] }
\end{myeq}
\noindent (see \cite[\S 2.8]{Bou03}), with the fibre of \w{h_k}  given by
\w[,]{\Omega^k \bF_k} where \w{\bF_k} is the fiber of the map
\w{\bW^{k} \to M^{k-1}\Wu} to the matching spectrum of \cite[X, \S 4.5]{BK72}.

Setting \w{\bX\sb{s}=\Tot\sb{s}\Wu} and \w[,]{\bK\sb{s}:=\Omega\sp{s}\bF\sb{s+1}}
we see that
\mydiagram[\label{eqadams}]{
  \bW^0\ar[d]^{j_0} & \Tot_1(\Wu) \ar[d]^{j_1} \ar[l]_(0.6){h_1} \cdots &
  \ar[l] \Tot_{k-1}(\Wu)\ar[d]^{j_{k-1}} & \Tot_k(\Wu)\ar[d]^{j_{k}} \ar[l]_(0.45){h_{k}} &
  \cdots\ar[l] \\ \bF_1 & \Omega \bF_2 & \Omega^{k-1}\bF_k & \Omega^k \bF_{k+1}
}
is an \ww{\bE}-Adams resolution for $\bY$.

Moreover,
\begin{myeq}\label{eqnormchain}
\bF_k~=~\bigcap\sb{j=0}\sp{k-1}\,\Ker(s\sp{j}:\bW^{k}\to\bW^{k-1})~,
\end{myeq}
As noted in \S \ref{rdegmod}, all the codegeneracies of \w{\Wu} are $\bE$-module maps,
so \w{\bF_{k}} is an $\bE$-module.

Moreover, the connecting homomorphism
\w{\delta\sp{k}:\pi\sb{\ast}\bF_k\to\pi\sb{\ast}\bF\sb{k+1}} for this tower of fibrations
is just the differential for the normalized cochains on \w{\pi\sb{\ast}\Wu} \wwh that is,
the alternating sum of the coface maps (see \cite[X, \S 6]{BK72}).

Given a (finite) spectrum $\bZ$, applying the functor \w{\F(\bZ,-)} to
\w{\Wu} yields a cosimplicial spectrum, whose total spectrum is the $\bE$-completion of
\w[,]{\F(\bZ,\bY)} under favorable assumptions. We define the $\bE$-\emph{based
  Adams spectral sequence} for \w{\F(\bZ,\bY)} to be the homotopy spectral sequence
for \w{\F(\bZ,-)} applied to \wref[,]{eqtottower} with
\begin{myeq}\label{spcseq}
E_1^{k,t} = \pi\sb{t-k}(\Omega^k\F(\bZ, \bF_{k}))~\cong~
\pi_0(\F(\Sigma^{t-k} \bZ, \Omega^k \bF_k))~\cong~\pi_0(\F(\Sigma^{t}\bZ, \bF_k))~.
\end{myeq}
(see \cite[X, \S 6]{BK72}). This agrees with the usual \ww{\bE}-based Adams
spectral sequence
from the \ww{E_{2}}-term on (see \cite[\S 2.2.4]{Rav86}, and compare \cite{BK72a}).
\end{mysubsection}

\begin{remark}\label{rdecomplma}
Note that by Theorem \ref{main1} \w{\Wu} (and thus our choice for the $\bE$-completion
of $\bY$), as well as the $\bE$-based Adams spectral sequence for $\bY$, are determined
functorially by \w{\fME\bY} (in fact, by \w{\rho\fME\bY} with its coalgebra structure)
and by $\bZ$, since the construction of \w{\Wu} in \S \ref{scsfx} is functorial in $\fX$.

The Reedy model category of cosimplicial $\bE$-modules of \cite[Theorem 4.1]{SShipA}
also has functorial factorizations, so the same remains true after fibrant replacement
of \w[.]{\Wu}
\end{remark}

%
%
\sect{Differentials in the Adams spectral sequence}
\label{cdiff}

In this section we assume $\bE$ is a ring spectrum, \w{\bY\in\Sp} is $\bE$-good,
and \w{\bZ\in\Sp} is finite and \w{\lambda\geq \wlY,\wlZ}
(in the notation of \S \ref{cardY}). We then let \w[,]{\fX=\fME\bY}
with \w{\bY\to\Wu} constructed from $\fX$ as in \S \ref{scsfx}, and identify the
$\bE$-based Adams spectral sequence for \w{\F(\bZ,\bY)} with the homotopy spectral
sequence of the cosimplicial spectrum \w[.]{\F(\bZ,\Wu)}
(We do not in fact need $\bY$ to be $\bE$-good in order for most of our results to hold,
but without some such assumption the spectral sequence need not converge, so information
about it will not be of much use.)

We can now state our first main result:

\begin{thm}\label{pdiff}
Given $\bE$, $\bZ$, and $\bY$ as above, for each \w[,]{r\geq 1} the
\ww{d\sb{r}}-differential in the $\bE$-based Adams spectral sequence for
\w[,]{\F(\bZ,\bY)} and thus its \ww{E\sb{r+1}}-term, can be calculated from
the cosimplicial \wwb{r-1}truncated space \w[.]{\Pnk{r-1}{0}\fME\bZ\lin{\Wu}}
\end{thm}

\begin{proof}
We recall the standard construction of the differentials in the homotopy
spectral sequence for the Tot tower of fibrations for \w[,]{\Xu:=\F(\bZ,\Wu)}
in terms of the interlocking long exact sequences of Figure \ref{fig1}.

%
%
\begin{figure}[htbp]
\begin{center}
\xymatrix@R=25pt@C=11pt{
\pi\sb{k+1}\Tot\sb{n}\Xu \ar[d]\sp{q\sp{n}} \ar[rr]\sp{\delta\sp{n}} &&
\pi\sb{k}\Omega\sp{n+1}N\sp{n+1}\Xu \ar[rr]\sp{j\sp{n+1}} &&
\pi\sb{k}\Tot\sb{n+1}\Xu \ar[d]\sp{q\sp{n+1}} \ar[rr]\sp(0.45){\delta\sp{n+1}} &&
\pi\sb{k-1}\Omega\sp{n+2}N\sp{n+2}\Xu\\
\pi\sb{k+1}\Tot\sb{n-1}\Xu \ar[rr]\sp{\delta\sp{n-1}} &&
\pi\sb{k}\Omega\sp{n}N\sp{n}\Xu \ar[rr]^{j\sp{n}} &&
\pi\sb{k}\Tot\sb{n}\Xu \ar[rr]\sp(0.45){\delta\sp{n}} &&
\pi\sb{k-1}\Omega\sp{n+1}N\sp{n+1}\Xu
}
\end{center}
\caption[fig1]{Exact couple for \w{\Tot} tower}
\label{fig1}
\end{figure}

Here the normalized chains for \w{\Xu} are given by
\w{N\sp{n}\Xu=\F(\bZ,\bF\sb{n})} (see \wref[).]{eqnormchain}

As we shall see below, the information needed to calculate the differentials at each
stage, consisting of various maps \w[,]{\bZ\to\bW\sp{i}} nullhomotopies thereof,
and so on:
\begin{enumerate}
\renewcommand{\labelenumi}{(\alph{enumi})}
\item can be expressed in terms of the mapping algebra \w{\fME\bZ} and the
  simplicial mapping algebra \w[;]{\fME\Wu}
\item in fact depends only on suitable truncations of these mapping algebras,
  if we are only calculating differentials up to the $r$-th stage.
\end{enumerate}

For this purpose, we think of the differential \w{d_r : E_r^{s,t} \to E_r^{s+r, t+r-1}}
as a ``relation'' (i.e., partially defined
map \w{E\sb{1}\sp{s,t} \to E\sb{1}\sp{s+r, t+r-1}} with a certain indeterminacy),
in the spirit of \cite{Bous89}.
Thus a class \w{\lra{\gamma}\in E_r^{n, n+k}} will be represented by an element
\w{\gamma \in E_1^{n, n+k}} such that \w[,]{d_1(\gamma)} $\cdots$ \w{d_{r-1}(\gamma)}
all have $0$ as a value.

In our interpretation, the value \w{[\beta]} we compute for the
differential \w{d\sb{j}} lies in \w{E_1^{n+j,
n+k+j-1}=\pi_0(\F(\Sigma^{n+k+j-1}\bZ, \bF_{n+j}))} (see
\wref[),]{spcseq} so its vanishing is witnessed by a choice of
nullhomotopy. This nullhomotopy takes value in a higher truncation
of the mapping algebra than the map $\beta$, which explains why
each successive differential requires a higher
  truncation\vsm .

\noindent\textbf{Step 1.}\hsm
Any class \w{\gamma\in E_1^{n, n+k}} is represented in turn by a map
\w[:]{\hat{g}:\Sigma^k \bZ \to \Omega^n \bF_n}
that is, a map \w{g:\Sigma^k\bZ \to \Tot_n {\Wu}} with \w{h_n \circ g=0} (see
\wref[).]{eqadams} By adjunction this defines a map of cosimplicial spectra
\w[.]{\Gu_n: \sk_n(\Du)\sb{+}\otimes \Sigma^k \bZ \to \Wu}
The value of the successive differentials \w{d_1(\gamma),\cdots,d_{r-1}(\gamma)}
serve as the successive obstructions to lifting \w{\Gu_n} to
\w[,]{\Gu_{n+1}: \sk_{n+1}(\Du)\sb{+}\otimes \Sigma^k \bZ \to \Wu} $\dotsc$
up to \w[.]{\Gu_{n+r-1}: \sk_{n+r-1}(\Du)\sb{+}\otimes \Sigma^k \bZ \to \Wu}

The cosimplicial map \w{\Gu_n: \sk_n(\Du)\sb{+}\otimes \Sigma^k \bZ\to\Wu}
consists of a sequence of maps of spectra
\w{G_n^j :\sk_n(\Delta[j])_+ \otimes \Sigma\sp{k}\bZ \to \bW^j} \wb[.]{j=0,1,\dotsc}
Since \w{\Wu} is Reedy fibrant,
\begin{myeq}\label{eqtotfibseq}
\Omega\sp{n}\bF\sb{n}~\to~\Tot\sb{n}\Wu~\xra{h\sb{n}}~\Tot\sb{n-1}\Wu
\end{myeq}
\noindent is a fibration sequence on the nose, so the fact that $\hat{g}$ lands
in \w{\Omega^n \bF_n} (and thus \w{G\sp{n}\sb{n}}
lands in \w[)]{\bF\sb{n}} implies that \w[.]{G_n^0=\cdots=G^{n-1}_n=0}
Moreover, \w{\sk_{n}\Delta[j]} is determined by \w{\Delta[j]}
and the coface maps in \w[,]{\Du}  for \w[,]{j>n} so the maps \w{G_n^j} \wb{j>n}
are determined by \w{G_n^n} and the coface maps of \w[.]{\Wu}

Note that \w{G_n^n} is adjoint to a map
\w[,]{\Sigma\sp{k}\bZ\to(\bW^n)\sp{\sk_n(\Delta[n])_+}} \wh in other words,
it is equivalent to a map \w{\widetilde{G}_n^n:\bS{0}\to\fME\Sigma\sp{k}\bZ\lin{\bM}} for
\w[,]{\bM:=(\bW^n)\sp{\sk_n(\Delta[n])_+}\in\TE} in terms of the
simplicial structure on $\bE$-modules (see \cite{SShipA})\vsm .

\noindent\textbf{Step 2.}\hsm
As noted above, $\gamma$ represents an element in \w{E_2} if
\w{d_1(\gamma) =0} in \w{E\sb{1}\sp{n+1,n+k}} \wh that is, if
\begin{myeq}\label{eqphi}
\phi~:=~\sum\sb{i=0}\sp{n}(-1)\sp{i}\,d\sp{i}\circ G\sb{n}\sp{n}
\end{myeq}
\noindent is nullhomotopic in
\w{\bF\sb{n+1}\subseteq\bW\sp{n+1}} (see Figure \ref{fig1}) The differential
\w{d_1(\gamma)} thus takes value in \w[\vsm.]{\pi\sb{0}\fME\bF\sb{n+1}}

\noindent\textbf{Step 3.}\hsm
By \wref[,]{eqnormchain} \w{\bF\sb{n+1}} is the (homotopy) limit of the
\w{3\times 3} diagram:
\mytdiag[\label{eqfibnorm}]{
  \bW\sp{n+1} \ar[rr]^{\top s\sp{j}}\ar[d] &&
  \prod\sb{j=0}\sp{n}\ \bW\sp{n}\ar[d] && \ast \ar[d] \ar[ll]\\
  \ast \ar[rr] && \ast && \ast \ar[ll]\\
    \ast \ar[rr] \ar[u] && \ast \ar[u] && \ast \ar[ll] \ar[u]
}
and similarly \w{P\bF\sb{n+1}} is the (homotopy) limit of the
\w{3\times 3} diagram:
\mytdiag[\label{eqfibs}]{
  (\bW\sp{n+1})\sp{\Delta[1]}\ar[d]^{\ev\sb{0}} \ar[rr]^{\top(s\sp{j})\sp{\Delta[1]}} &&
  \prod\sb{j=0}\sp{n}\ (\bW\sp{n})\sp{\Delta[1]} \ar[d]^{\prod\ev\sb{0}} &&
  \ast\sp{\Delta[1]}=\ast \ar[ll] \ar[d]\\
 \bW\sp{n+1} \ar[rr]^{\top s\sp{j}} && \prod\sb{j=0}\sp{n}\ \bW\sp{n} && \ast \ar[ll]\\
\ast \ar[u] \ar[rr] && \ast \ar[u] && \ast \ar[ll]\ar[u]
}
We have a map from \wref{eqfibs} to \wref{eqfibnorm}
induced by \w[,]{\ev\sb{1}} and by taking limits we obtain the path fibration
\w[.]{p:P\bF\sb{n+1}\to\bF\sb{n+1}}

Thus the path-loop fibration sequence for \w{\bF\sb{n+1}} is obtained by taking
iterated pullbacks of diagrams built from \w{\bW\sp{n}} and \w[,]{\bW\sp{n+1}}
first vertically, and then horizontally (see \cite[XI, 4.3]{BK72}).
We therefore see that both the class $\phi$ of \wref{eqphi} representing
\w{d_1(\gamma)} in \w[,]{\pi\sb{0}\fME\Sigma\sp{k}\bZ\lin{\bF\sb{n+1}}} and our
choice of a nullhomotopy $\Phi$ for it, are determined, according to
\cite[Theorem 10]{MathP}, by various compatible maps and nullhomotopies into the diagrams
\wref{eqfibs} and \wref[.]{eqfibnorm}

These maps and nullhomotopies, respectively,  correspond to maps and nullhomotopies,
respectively, from \w{\bS{0}} to
\w{\Pnk{1}{0}\fME\bZ\lin{\bW\sp{n+1}}} and \w[,]{\Pnk{1}{0}\fME\bZ\lin{\bW\sp{n}}}
(composed with
\w[)]{(s\sp{j})\sb{\ast}:\fME\Sigma\sp{k}\bZ\lin{\bW\sp{n+1}}\to
  \fME\Sigma\sp{k}\bZ\lin{\bW\sp{n}}} \wwh
which can be expressed in terms of the truncated mapping algebra
\w{\Pnk{1}{0}\fME\Sigma\sp{k}\bZ}
and the action on it of the free simplicial truncated mapping algebra
\w{\Pnk{1}{0}\fME\Wu}  (in the sense of \wref[)]{eqaction} \wwh in other words,
in terms of the $1$-truncated cosimplicial space
\w[.]{\Pnk{1}{0}\fME\Sigma\sp{k}\bZ\lin{\Wu}}

By a standard argument in the long exact sequence of the fibration \wref[,]{eqtotfibseq}
we can use $\Phi$ to extend \w{\Gu_n} to a map
\w[.]{\Gu\sb{n+1}: \sk\sb{n+1}(\Du)\sb{+}\otimes \Sigma^k \bZ \to \Wu}
Note that this is determined by
\w[,]{G\sp{n}\sb{n+1}:\Sigma^k \bZ \to(\bW\sp{n})\sp{\Delta[n]\sb{+}}}
\w[,]{G\sp{n+1}\sb{n+1}:\Sigma^k \bZ \to(\bW\sp{n+1})\sp{\Delta[n+1]\sb{+}}}
and the maps between them coming from the coface maps of \w[.]{\Wu}

Because \w{G\sp{j}\sb{n+1}=0} for \w[,]{j<n} the maps actually land in
\w{\Omega\sp{n}\bW\sp{n}} and \w[,]{\Omega\sp{n}\bW\sp{n+1}} respectively, so they
take value in
\w[\vsm.]{\Pnk{1}{0}\fME\Sigma\sp{k}\bZ\lin{\Omega\sp{n}\Wu}}

\noindent\textbf{Step 4.}\hsm
Assume by induction that, for \w[,]{r\geq 1} $\gamma$ represents an element in
\w[,]{E\sb{r}} so the differentials on $\gamma$ up to \w{d\sb{r-1}} vanish, and
we have an extension of \w{\Gu_n} to
$$
\Gu\sb{n+r-1}: \sk\sb{n+r-1}(\Du)_+ \otimes \Sigma^k \bZ~\to~\Wu
$$
with \w{G\sb{n+r-1}^j =0} for \w{0 \leq j \leq n-1}
(and again for \w[,]{j>n+r-1} \w{G\sb{n+r-1}^j} is determined by
\w{G\sb{n+r-1}\sp{n+r-1}} and the coface maps of \w[).]{\Wu}
As usual, we can extend this further to \w{\Gu\sb{n+r}} (for \emph{some} choice
of \w[)]{\Gu\sb{n+r-1}} if and only if \w{d\sb{r}(\gamma)} vanishes.

The map \w{\Gu\sb{n+r-1}} represents a class
\w{\alpha\sb{r-1}} in \w{\pi\sb{k}\Tot\sb{n+r-1}\Xu}
(as in Figure \ref{fig1}). Applying the connecting homomorphism
$$
\delta\sp{n+r-1}:\pi\sb{k}\Tot\sp{n+r-1}\Xu\to
\pi\sb{k-1}\Omega\sp{n+r}N\sp{n+r}\Xu
$$
to \w{\alpha\sb{r-1}} yields a class
\w[,]{[\beta\sb{r-1}]\in[\Sigma\sp{k-1}\bZ,\,\Omega\sp{n+r}\bF\sb{n+r}]} which represents
the value of \w[.]{d\sb{r}(\gamma)}

Note that \w{\beta\sb{r-1}} (as a map into \w[)]{\bF\sb{n+r}} is represented in turn as
in \wref{eqfibnorm} above by a map of spectra
  \w[,]{\widehat{b}\sb{r-1}:\Sigma\sp{k}\bZ\to\Omega\sp{n+r-1}\bW\sp{n+r}}
and thus by
\w{b\sb{r-1}\in(\Pnk{r-1}{0}\fME \Sigma\sp{k}\bZ\lin{\Omega\sp{n}\bW\sp{n+r}})\sb{r-1}}
(an \wwb{r-1}simplex in the simplicial set \w[,]{\Pnk{r-1}{0}(-)}
as in \S \ref{window}).

Our earlier choices of \w[,]{\Gu\sb{n+r-2}} $\dotsc$, \w[,]{\Gu\sb{n}} also come into
the picture in the form of (iterated) coface maps of \w{\Wu} applied to earlier
simplices \w[,]{\beta\sb{r-2}} $\dotsc$, \w[.]{\beta\sb{1}} This is why we need all
of \w[,]{\Pnk{r-1}{0}\fME \Sigma\sp{k}\bZ \lin{\Omega\sp{n}\Wu}} and not just its
\wwb{r-1}simplices. See \cite[\S 5]{BBS17} for an explicit description of the
combinatorics in a slightly different formulation (which is not needed here).

We thus see by induction that the choice of \w[,]{\Gu\sb{n+r-1}}
as well as the value of \w[,]{d\sb{r}(\gamma)} may be expressed
in terms of \w[.]{\Pnk{r-1}{0}\fME \Sigma\sp{k}\bZ \lin{\Omega\sp{n}\Wu}}

If \w{d\sb{r}(\gamma)} vanishes, for some collection of choices as above, the map
\w{\beta\sb{r-1}} is nullhomotopic; as in Step 3, the choice of a nullhomotopy \wh
and thus the lift of \w{\Gu\sb{n+r-1}} to \w{\Gu\sb{n+r}} and the resulting value
of \w{d\sb{r+1}(\gamma)} \wwh is encoded one simplicial dimension higher \wh that is,
in the cosimplicial space \w[.]{\Pnk{r}{0}\fME \Sigma\sp{k}\bZ\lin{\Omega\sp{n}\Wu}}

Finally, note that up to homotopy the mapping algebra \w{\fME\Sigma\sp{k}\bZ} is just
\w[,]{\Omega\sp{k}\fME\bZ} since it is a homotopy spectral functor, and for the
same reason \w[.]{\fME\Sigma\sp{k}\bZ \lin{\Omega\sp{n}\Wu}\simeq
  \Omega\sp{n}\fME\Sigma\sp{k}\bZ \lin{\Wu}}
\end{proof}

\begin{mysubsection}{Resolution model categories}
\label{srmc}
Since any spectrum is a homotopy group object in \w[,]{\Sp}
from Lemma \ref{yoneda} we see that for all \w[,]{\bM \in \TE} the free spectral functor
\w{\fME\bM} is a homotopy cogroup object in \w[.]{\ST}

Thus by \cite[Theorem 2.2.]{Jar04}:

\begin{enumerate}
\renewcommand{\labelenumi}{(\alph{enumi})}
\item There is a resolution model category structure on
\w[,]{(\ST)\sp{\Delta\op}=\Sp\sp{\TE\times\Delta\op}} in which the weak equivalences
are the $E^2$-equivalences (cf.\ \S \ref{detwe}): that is, maps \w{f : \fUd \to \fWd}
of simplicial spectral functors such that for each \w{\bM \in \TE} the induced map
\w{\pi\sb{0}\fUd\lin{\bM}\to\pi\sb{0}\fWd\lin{\bM}} of simplicial groups is
a weak equivalence.
\item Similarly, for each\w[,]{\bM \in \TE} any fibrant and cofibrant replacement
for \w{\fME\bM}  in the \ww{\Pnk{r}{0}}-model structure on \w{\ST}
is a homotopy cogroup object there,
so by Proposition \ref{pmodelcat}, \w{(\ST)\sp{\Delta\op}}
also has a \w{\Pnk{r}{0}} resolution model category structure,
with the same $E^2$-equivalences.
\item Finally, given a cosimplicial $\bE$-module \w[,]{\Wu} 
let \w{\Thw} denote the simplicially enriched category whose objects are
\w{\bW\sp{i}} \wb{i=0,1,2\dotsc} with truncated simplicial mapping spaces
\w{\umap(\bW\sp{i},\bW\sp{j}):=\Pnk{r}{0}\F(\bW\sp{i},\bW\sp{j})\sb{0}} as in
\S \ref{window}. The category \w{\Spa\sp{\Thw}} of simplicial functors (with
respect to \w[)]{\umap} also has a proper model category structure
(see \cite[\S 1.23]{BBC19}), and from Lemma \ref{lyoneda} we see that
\w{\Pnk{r}{0}\fME\bW\sp{i}} is a cogroup object in \w[,]{\Spa\sp{\Thw}} so we
get a corresponding resolution model category structure on the simplicial
objects \w{(\Spa\sp{\Thw})\sp{\Delta\op}} (see \cite[\S 2.12]{BBC19}).
\end{enumerate}

(The cosimplicial spectrum we actually have in mind in (c), in the context of
the proof of Theorem \ref{pdiff}, is \w[).]{\Omega\sp{n}\Wu}
\end{mysubsection}

We now have:

\begin{prop}\label{E2_equi}
Let \w{\Wu} be constructed from $\bY$ as in \S \ref{scsby} and assume \w{\fUd} is any
resolution of \w{\fX=\fME\bY} (that is, a cofibrant replacement,
in the model category structure of \S \ref{srmc}(b), for the simplicial
spectral functor \w{c\sb{\bullet}(\fX)} which is $\fX$ in each simplicial dimension);
then \w{\Pnk{r}{0}\fME\Wu} is
$E^2$-equivalent to \w[.]{\Pnk{r}{0}\fUd}
\end{prop}

\begin{proof}
Since
$$
\pi_j \Pnk{r}{0} \fX\lin{\bM}\cong
  \begin{cases}\pi_j \fX\lin{\bM} &\text{for}\ 0\leq j \leq r\\0 & \text{otherwise,}
  \end{cases}
$$
  this follows from Proposition \ref{acyclic_reso}, and the fact that
  \w{\Pnk{r}{0}\fME\Wu} is a resolution of \w{\Pnk{r}{0}\fX} in the
  model category structure of \S \ref{srmc}(c).
\end{proof}

From Theorem \ref{pdiff} and Proposition \ref{E2_equi} we deduce:

\begin{thm}\label{main2}
If \w{\bE=\bH R} for a commutative ring $R$, $\bZ$ is a fixed finite spectrum, and
$\bY$ is a $\bE$-good spectrum, then for any \w{r\geq 0} the
\Elt{r+2} of the $\bE$-based Adams spectral sequence for \w{\F(\bZ,\bY)} is
determined by the truncated mapping algebra \w[.]{\Pnk{r}{0}\fME\bY}
\end{thm}

\begin{proof}
Let \w{\fUd} be any resolution of \w{\fME\bY} in the model category structure
of \S \ref{srmc}(b) (which depends only on \w[,]{\Pnk{r}{0}\fME\bY} up to
$E^2$-equivalence).
By   \cite[Theorem 3.21\textit{ff.}]{BBC19}, we can construct a cosimplicial resolution
\w{\Uu} of $\bY$ in the resolution model category structure on \w{\Sp\sp{\Delta}} of
\cite[\S 3]{Bou03}, such that \w{\fME\Uu} is Reedy equivalent to \w{\fUd} (that is,
there is a map of simplicial spectral functors \w{\ff:\fME\Uu\to\fUd} with each
\w{\ff\sb{n}:\fME\bU\sp{n}\to\fU\sb{n}} a weak equivalence of spectral functors).

Thus the truncated cosimplicial space \w{\Pnk{r}{0}\fME\bZ\lin{\Uu}} is well defined
up to Reedy weak equivalence. Moreover, there is an $E^2$-equivalence
\w{\fg:\Uu\to\Wu} (where \w{\Wu} is the cosimplicial spectrum  of \S \ref{scsfx}),
  which induces an $E^2$-equivalence of truncated cosimplicial spaces
  \w[,]{\Pnk{r}{0}\fME\bZ\lin{\Uu}\to\Pnk{r}{0}\fME\bZ\lin{\Wu}} and thus a
  map of spectral sequences which is an isomorphism form the \Ett on.
The result then follows from Theorem \ref{pdiff}.
\end{proof}

This presumably holds for any ring spectrum $\bE$, though the results of
\cite[\S 3]{BBC19} are only known for \w[.]{\bH R}
 
\begin{remark}\label{rsummary}
Our main goal here was to show what sort of general information
about $\bE$-modules, combined with what specific data on $\bY$ and
$\bZ$, suffice to determine the \Elt{r} of the $\bE$-based Adams
spectral sequence for \w{\F(\bZ,\bY)} \wwh modelled on the way the
\Elt{2} is a functor of \w{E\sp{\ast}\bY} (under favorable
assumptions on $\bE$).

As Theorems \ref{pdiff} and \ref{main2} show, the necessary data can be described in
the language of truncated mapping algebras, our main object of study here.
For \w[,]{\bE=\HFp} \w[,]{\bZ=\bS{0}} and \w[,]{r=2} this data reduces to the knowledge
of \w{H\sp{\ast}(\bY;\Fp)} as a module over the Steenrod algebra, as in \cite{Ada58}.
\end{remark}


\begin{thebibliography}{MMSS}
%
\bibitem[A]{Ada58}
J.F.~Adams,
``On the structure and applications of the Steenrod algebra'',\hsm
\textit{Comm.\ Math.\ Helv.} \textbf{32} (1958), pp.~180-214.
%
\bibitem[BBS]{BBS17}
S.~Basu, D.~Blanc, \& D.~Sen,
``Higher structures in the unstable Adams spectral sequence'',\hsm
preprint, 2017.
%
\bibitem[BBC]{BBC19}
H-J.~Baues, D.~Blanc, \& B.~Chorny,
``Truncated derived functors and spectral sequence'',\hsm
\textit{Homology Homotopy Appl.}, to appear.
%
\bibitem[BJ]{BJ11}
H.-J.~Baues \& M.A.~Jibladze,
``Dualization of the Hopf algebra of secondary cohomology operations
and the Adams spectral sequence'',\hsm
\textit{J.\ $K$-Theory} \textbf{7} (2011), pp.~203-347.
%
\bibitem[BS]{BS18}
D.~Blanc \& D.~Sen,
``Mapping spaces and $R$-completion'',\hsm
\textit{J.\ Hom.\ Rel.\ Str.} \textbf{13} (2018), pp.~635-671.
%
\bibitem[B1]{Bou79}
A.K.~Bousfield,
``The localization of spectra with respect to homology'',\hsm
\textit{Topology} \textbf{18} (1979), pp.~257-281
%
\bibitem[B2]{Bous89}
A.K.~Bousfield,
``Homotopy Spectral Sequences and Obstructions'',\hsm
\textit{Israel J.\ Math.}, \textbf{66} (1989), pp.~54-104.
%
\bibitem[B3]{Bou01}
A.K.~Bousfield,
``On the telescopic homotopy theory of spaces'',\hsm
\textit{Trans.\ AMS} \textbf{353} (2001), pp.~2391-2426.
%
\bibitem[B4]{Bou03}
A.K.~Bousfield,
``Cosimplicial resolutions and homotopy spectral sequences in model categories'',\hsm
\textit{Geom.\ Top.} \textbf{7} (2003), pp.~1001-1053.
%
\bibitem[BK1]{BK72}
A.K.~Bousfield \& D.M.~Kan,
\emph{Homotopy limits, completions and localizations},
Sprin\-ger \textit{Lec.\ Notes Math.} \textbf{304}, Berlin-\-New York, 1972.
%
\bibitem[BK2]{BK72a}
A.K.~Bousfield \& D.M.~Kan,
``The homotopy spectral sequence of a space with coefficients in a ring'',\hsm
\textit{Topology} \textbf{11} (1972), pp.~79-106.
%
\bibitem[C]{Cho16}
B.~Chorny,
``A classification of small homotopy functors from spectra to spectra'',\hsm
\textit{Fund.\ Math.} \textbf{234} (2016), pp.~101-125.
%
\bibitem[Ha]{HasseG}
M.~Hasse,
``Einige Bemerkungen \"{u}ber Graphen, Kategorien und Gruppoide'',\hsm
\textit{Math.\ Nach.} \textbf{22} (1960), pp.~255-270.
%
\bibitem[Hi]{Hir03}
P.S.~Hirschhorn,
\textit{Model Categories and their Localizations},\hsm
Math.\ Surveys \& Monographs \textbf{99}, AMS, Providence, RI, 2002.
%
\bibitem[HSS]{HSS}
M.~Hovey, B.~Shipley, \& J.H.~Smith,
``Symmetric spectra'',\hsm
\textit{J.\ AMS} \textbf{10} (2000), pp.~149-208.
\bibitem[J]{Jar04}
J.F.~Jardine,
``Bousfield's $E\sb{2}$ Model Theory for Simplicial Objects'',\hsm
in P.G.~Goerss \& S.B.~Priddy, eds., \textit{Homotopy Theory: Relations with Algebraic
  Geometry, Group Cohomology, and Algebraic $K$-Theory}, AMS, Providence, RI, 2004,
pp.~305-319
%
\bibitem[K]{Kel82}
G.M.~Kelly,
\textit{Basic concepts of enriched category theory}, \hsm
Cambridge U.\ Press, Cambridge, UK, 1982.
%
\bibitem[MMSS]{MMSS01}
M.A.~Mandell, J.P.~May, S.~Schwede, \& B. Shipley,
``Model categories of diagram spectra'',\hsm
\textit{Proc.\ London Math.\ Soc.\ (3)} \textbf{82} (2001), pp.~441-512.
%
\bibitem[M]{MathP}
M.~Mather,
``Pull-backs in homotopy theory'',\hsm
\textit{Can.\ J.\ Math.} \textbf{28} (1976), pp.\ 225-263.
%
\bibitem[N]{Nov67}
S.P.~Novikov,
``Operation rings and spectral sequences of the Adams type
            in extraordinary cohomology theories'',\hsm
\textit{Dokl.\ Akad.\ Nauk SSSR} \textbf{172} (1967) pp.~33-36.
%
\bibitem[R]{Rav86}
D.C.~Ravenel,
\textit{Complex Cobordism and Stable Homotopy Groups of Spheres},\hsm
Academic Press, Orlando, 1986.
%
\bibitem[SS]{SShipA}
S.~Schwede \& B.E.~Shipley,
``Algebras and modules in monoidal model categories'',\hsm
\textit{Proc.\ London Math.\ Soc.\ (3)} \textbf{80} (2000), pp.~491-511.
%
\bibitem[S]{Sto90}
C.R.~Stover,
``A Van Kampen spectral sequence for higher homotopy groups'',\hsm
\textit{Topology} \textbf{29} (1990), pp.~9-26.
%
\bibitem[W]{Wei94}
C.A.~Weibel,
\textit{An Introduction to Homological Algebra},\hsm
Cambridge U.\ Press, Cambridge, UK, 1994.
%
\end{thebibliography}
\end{document}